\documentclass[10pt]{article}
\usepackage[a4paper]{geometry}
\usepackage{amsmath, amssymb, amsthm
}

\usepackage[colorlinks]{hyperref}
\hypersetup{linkcolor=black,citecolor=black,filecolor=black,urlcolor=blue}

\usepackage{appendix}
\usepackage{multicol} 
\usepackage{makeidx}
\usepackage{cite}
\usepackage{captdef}
\makeindex
\usepackage{paralist}

\usepackage{psfrag}
\usepackage{graphicx}
\usepackage{graphics}
\usepackage{color}

\usepackage{xcolor}
\usepackage{afterpage}

\usepackage{subcaption}

\usepackage[showonlyrefs]{mathtools}  
\mathtoolsset{showonlyrefs=false}

\newcommand{\R}{\mathbb{R}}

\newtheorem{theo}{Theorem}[section]
\newtheorem{coro}[theo]{Corollary}
\newtheorem{lemma}[theo]{Lemma}
\newtheorem{prop}[theo]{Proposition}
\theoremstyle{definition}

\newtheorem{remark}[theo]{Remark}

\newcommand{\eps}{\varepsilon}

\newcommand{\N}{{\mathbb{N}}}

\renewcommand{\epsilon}{\varepsilon}

\renewcommand{\leq}{\leqslant}
\renewcommand{\geq}{\geqslant}

\numberwithin{equation}{section}

\def\sideremark#1{\ifvmode\leavevmode\fi\vadjust{\vbox to0pt{\vss
 \hbox to 0pt{\hskip\hsize\hskip1em
 \vbox{\hsize2.1cm\tiny\raggedright\pretolerance10000
  \noindent #1\hfill}\hss}\vbox to15pt{\vfil}\vss}}}%

\title{Sharp concentration estimates near criticality for radial sign-changing solutions of Dirichlet and Neumann problems}
\author{Massimo Grossi\footnote{Dipartimento di Matematica, Universit\`{a} di Roma La Sapienza, P.le A. Moro 2, 00185 Roma, Italy; massimo.grossi@uniroma1.it},\ \  Alberto Salda\~{n}a,\footnote{Institut f\"ur Analysis, Karlsruhe Institute for Technology, Englerstra\ss e 2, 76131, Karlsruhe, Germany; alberto.saldana@partner.kit.edu}\ \ \& Hugo Tavares\footnote{CMAFCIO \& Departamento de Matem\'atica, Faculdade de Ci\^encias da Universidade de Lisboa, Edif\'icio C6, Piso 1, Campo Grande 1749-016  Lisboa, Portugal; hrtavares@ciencias.ulisboa.pt}}

\date{\today}

\begin{document}

\maketitle

\begin{abstract}
We consider radial solutions of the slightly subcritical problem $-\Delta u_\eps = |u_\eps|^{\frac{4}{n-2}-\varepsilon}u_\eps$ either on $\R^n$ ($n\geq 3$) or in a ball $B$ satisfying Dirichlet or Neumann boundary conditions.  In particular, we provide sharp rates and constants describing the asymptotic behavior (as $\eps\to 0$) of all local minima and maxima of $u_\eps$ as well as its derivative at roots.  Our proof is done by induction and uses energy estimates, blow-up/normalization techniques, a radial pointwise Pohozaev identity, and some ODE arguments. As corollaries, we complement a known asymptotic approximation of the Dirichlet nodal solution in terms of a tower of bubbles and present a similar formula for the Neumann problem.

\end{abstract}

\section{Introduction}
Let $n\geq 3$, $\eps\in(0,\frac{4}{n-2})$, and consider the subcritical problem in the whole space
\begin{align}\label{ws}
 -\Delta u_\eps = |u_\eps|^{\frac{4}{n-2}-\varepsilon}u_\eps\quad\text{ in }\R^n
\end{align}
or in a unitary ball $B$ centered at the origin
\begin{align}\label{sceq}
 -\Delta u_\eps = |u_\eps|^{\frac{4}{n-2}-\varepsilon}u_\eps\quad\text{ in }B
\end{align}
with either Dirichlet boundary conditions
\begin{align}\label{dbc}
u_\eps=0\  \text{ on }\partial B
\end{align}
or Neumann boundary conditions 
\begin{align}\label{nbc}
\partial_\nu u_\eps=0\ \text{ on }\partial B.
\end{align}
We are interested in a precise description of the asymptotic behavior of classical radially symmetric solutions as $\eps\to 0$. 

\medskip

The study of this type of results for the \emph{Dirichlet problem} originated in \cite{AP87}, where using ODE arguments (Emden-Fowler theory) it is proved that the unique \emph{positive} radial solution of \eqref{sceq}, \eqref{dbc} satisfies
\begin{align}
 \lim_{\eps\to 0} \eps u_\eps^2(0) &= (n(n-2))^{\frac{n-2}{2}}\frac{4\Gamma(n)}{(n-2)\Gamma(\frac{n}{2})^2},\label{AP1}\\
 \lim_{\eps\to 0} \eps^{-\frac{1}{2}} u_\eps(x) &= (n(n-2))^{\frac{n-2}{4}}
\Big( \frac{(n-2)\Gamma(\frac{n}{2})^2}{4\Gamma(n)} \Big)^\frac{1}{2}
 (|x|^{2-n}-1)\qquad \text{for $|x|\in(0,1]$}\label{AP2}.
\end{align}
In particular, \eqref{AP1}, \eqref{AP2} show that $u_\eps$ \emph{concentrates} at the origin and $u_\eps\to 0$ uniformly in compact subsets of $B\backslash \{0\}$ as $\eps\to 0$.  This behavior is consistent with the well-known fact that \eqref{sceq}, \eqref{dbc} has no nontrivial solutions (in any star-shaped domain) for $\eps=0$, due to the Pohozaev identity \cite{P65}. 

In \cite{BP89} the question of the asymptotic behavior of $u_\eps$ is revisited (considering also the operator $-\Delta-\lambda$ for $\lambda\geq 0$) for $n=3$, and the estimates \eqref{AP1}, \eqref{AP2} are shown using PDE methods based on the Pohozaev identity, Green functions, and elliptic regularity theory.

The authors in \cite{BP89} conjectured that a similar asymptotic behavior as in \eqref{AP1}, \eqref{AP2} should also occur for positive nonradial solutions in general domains (note that, by a moving-plane argument, all positive solutions of \eqref{sceq},\eqref{dbc} are radial \cite{GNN79}).  This conjecture was proved independently in \cite{H91} and \cite{R89}, where it is shown that \emph{least-energy solutions} possess the same limiting behavior with constants depending on the associated Green function. Furthermore, in \cite[Proposition 1]{H91} a $C^{1,\alpha}$-convergence of the solution at the boundary of the domain is shown; in particular in the case of a ball, by \eqref{AP2}, this implies that 
\begin{align}\label{AP2p}
 \lim_{\eps\to 0} \eps^{-\frac{1}{2}} u'_\eps(1) &=-
 (n-2)(n(n-2))^{\frac{n-2}{4}}
\Bigg( \frac{(n-2)\Gamma(\frac{n}{2})^2}{4\Gamma(n)} \Bigg)^\frac{1}{2},
 \end{align}
where  $u_\eps$ is the positive solution of \eqref{sceq},\eqref{dbc}.

After these seminal works, many extensions for related problems have been studied, for example, for Dirichlet positive solutions of operators of the form $-\Delta u +a(|x|)u=f(|x|)u^{\frac{n+2}{n-2}-\eps}$; the number of interesting papers for this problem is too large to give here a complete list of references, so just to give a glimpse of the results in this direction we refer to \cite{R02} and the references therein.

The first paper to consider Dirichlet \emph{sign-changing} subcritical solutions with \emph{annular-shaped nodal domains} seems to be \cite{PW07}, where the authors use a Lyapunov-Schmidt reduction scheme in domains with symmetries (general domains were considered afterwards in \cite{MP10}). See also \cite{CD06} where the particular case of a ball is considered using a similar strategy as in \cite{PW07}.  Particularly important for our approach is the recent paper \cite{DIP17}, where the Morse index of radial solutions of the slightly subcritical problem is studied using energy methods.  We describe in more detail the results of \cite{PW07,CD06,DIP17} and the relationship with ours after Theorem \ref{pointwise:coro}. 

Concerning nodal solutions of \emph{variants} of \eqref{sceq}, the literature is again very extensive. The following is an incomplete list of references whose only aim is to show the diversity of results and techniques used in this direction, see \cite{BMP06,BDP13,BDP13b,MP10,PW07,ABP90,I15} and the references therein.

\medbreak

The \emph{Neumann problem} has been much less studied.  A first remark is that all nontrivial solutions of \eqref{sceq}, \eqref{nbc} are necessarily sign-changing, since
\begin{align*}
0=-\int_B\Delta u_\eps = \int_B|u_\eps|^{\frac{4}{n-2}-\varepsilon}u_\eps.
\end{align*}
From a variational point of view, least-energy solutions have been considered in the sublinear ($\eps\in(\frac{4}{n-2},\frac{4}{n-2}+1)$) case  \cite{PW15}, in the superlinear-subcritical ($\eps\in(0,\frac{4}{n-2})$) case  \cite{ST17}, and in the critical ($\eps=0$) case \cite{CK91} (see also \cite{CK90}). In contrast to the Dirichlet problem, the Neumann b.c. allow the existence of nontrivial solutions if $\eps=0$ and, due to a symmetry-breaking phenomenon, minimal-energy solutions are not radially symmetric \cite{PW15,ST17,CK91}, and they do not blow up as $\eps\to 0$ (we study in detail the qualitative properties of this kind of solutions in a forthcoming paper). For $\eps=0$ the are no radial solutions to \eqref{sceq}, \eqref{nbc}, since the change of sign would force the existence of an interior nodal sphere; nevertheless, for $\eps>0$, \emph{radial} solutions of \eqref{sceq}, \eqref{nbc} do exist and therefore it is natural to study the asymptotic behavior of Neumann solutions as $\eps\to 0$, and we do this in Theorem \ref{explicit:thm_Neumann} below, where we detail its concentration rates.

In the literature one can also find an ample study of Neumann \emph{positive} solutions of $-\Delta u=\lambda u + u^\frac{n+2}{n-2}$ with $\lambda\neq 0$; here again we only mention the survey paper \cite{C07} and the references therein to have an idea of this topic, but we emphasize that many other interesting papers studying this problem are available.

Equation \eqref{ws} is not only a mathematical paradigm in nonlinear analysis of PDEs, but also it has a physical motivation, since, for $n=3$, radial solutions of \eqref{ws} solve the \emph{Lane-Emden equation of index $s=5-\varepsilon$}, namely,
\begin{align}\label{lem}
 \frac{1}{\xi^2}\frac{d}{d\xi}\left( \xi^2\frac{d}{d\xi}\theta \right)=-\theta^{s},\qquad \theta(0)=1,\qquad \theta'(0)=0
\end{align}
(usually only positive values of $\theta$ are considered). Equation \eqref{lem} is used in astrophysics to model self-gravitating spheres of plasma, such as stars or self-consistent stellar systems in polytropic-convective equilibrium, where the pressure $P$ and the density $\rho$ ($=k\theta^s$) satisfy a nonlinear relationship $P=c\rho^\frac{s+1}{s}$, see \cite{C57}.  In this setting, a solution $\theta$ is often called a \emph{polytrope}, and it contains (up to constants) important physical information, such as the radius of the star (the first root $r_1$ of $\theta$), the total mass ($\int_{B_{r_1}}\theta^s$), the pressure ($\theta^{s+1}$), and (for an ideal gas) the temperature is proportional to $\theta$.

\bigskip

The aim of this paper is to extend and generalize the results from \cite{AP87} to any radial solution of \eqref{ws} and of \eqref{sceq} satisfying either \eqref{dbc} or \eqref{nbc} with an arbitrary number of interior nodal spheres.  In particular, we show similar estimates to \eqref{AP1}, \eqref{AP2} and describe in a precise way the asymptotic behavior of all critical points, of all roots, and the values of $u$ and $u'$ at these points respectively. Here, a very delicate and precise analysis is needed to obtain the explicit rates and the associated constants. Our results seem to be the first to consider explicit constants regarding sign-changing solutions and the exact asymptotic behavior of the Neumann problem. As shown below, it turns out that the asymptotic study of the Dirichlet and the Neumann problem is closely intertwined and intimately related to the subcritical problem in the whole space $\R^n$.    

To state our main results let us  introduce some notations.  For a radial function $w:\R^n\to\R$ we use a common abuse of notation and identify $w(r)$ with $w(x)$ for $|x|=r$. Then, we say that $x\mapsto w(x)$ has \emph{$k$-interior zeros} if $r\mapsto w(r)$ has $k$ interior zeros in $(0,1)$. Given $m\in \N$, it is known (see for example \cite{N83}) that there is a unique (up to a sign) radial solution of the Dirichlet problem \eqref{sceq}, \eqref{dbc}  with exactly $m-1$ interior zeros and, for $m\geq 2$, a unique (up to a sign) radial solution of the Neumann problem \eqref{sceq}, \eqref{nbc} with exactly $m-1$ interior zeros.  Moreover, between two consecutive zeros there is exactly one critical point, which is either a local minimum or maximum.  For radial solutions of \eqref{sceq}, throughout this paper we use $(\delta_{k,\eps})_{k\in\N}$ and $(\rho_{k,\eps})_{k\in\N}$ to denote \emph{decreasing} sequences of critical points and zeros in $[0,1]$, respectively, see Figure \ref{fig2} (Dirichlet case) and Figure \ref{fig1} (Neumann case).

  \begin{figure}[h!]
  \begin{center}
  \includegraphics[width=.80\textwidth]{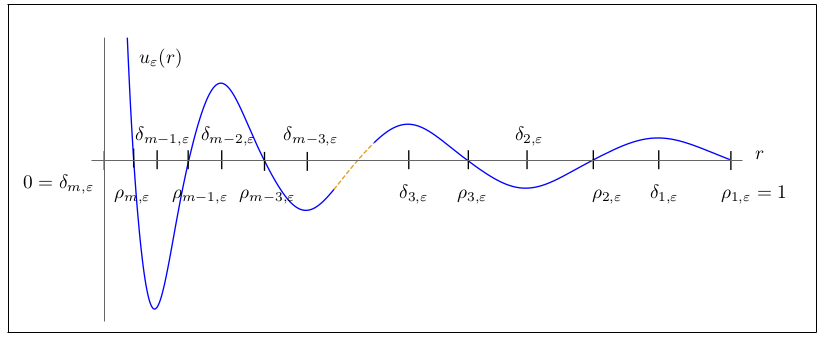}
\end{center}
 \caption{Example of a radial Dirichlet solution of \eqref{sceq} with $m-1(=6)$ interior zeros.}
 \label{fig2}
 \end{figure}

 \begin{figure}[h!]
  \begin{center}
 \includegraphics[width=.80\textwidth]{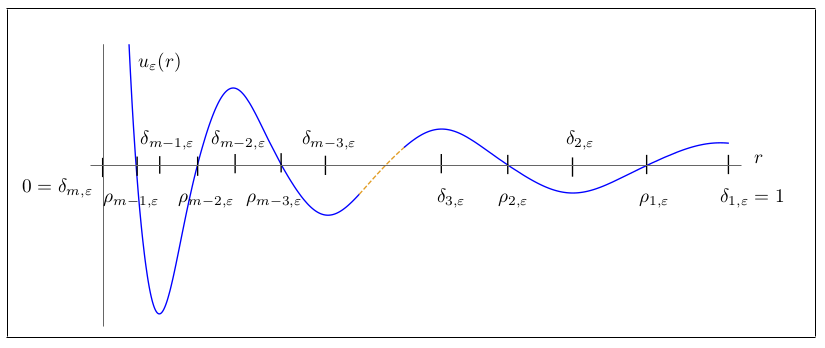}
 \end{center}
\caption{Example of a radial Neumann solution of \eqref{sceq} with $m-1(=6)$ interior zeros.}
\label{fig1}
\end{figure}

Let $\Gamma$ denote the standard Gamma function (which satisfies $\Gamma(n)=(n-1)!$ if $n\in\N$) and set 
\begin{align}\label{Cchi}
 \kappa_n:=\frac{(n-2)}{4}\frac{\Gamma(\frac{n}{2})^2}{\Gamma(n)}.
  \end{align}

 Our main results are the following.

\begin{theo}[Dirichlet case]\label{explicit:thm_Dirichlet}
For $n\geq 3$, $m\in\N$, and $\eps\in (0,\frac{4}{n-2})$, let $u_\eps$ be a radial solution of \eqref{sceq}, \eqref{dbc} with $m-1$ interior zeros. Let $(\delta_k)_{k=1}^m$ be the decreasing sequence of all the critical points of $u_\eps$ in $[0,1]$ and $(\rho_k)_{k=1}^{m}$ the decreasing sequence of all the zeros of $u_\eps$ in $[0,1]$. Then
\begin{align*}
0=\delta_{m,\eps}<\rho_{m,\eps}<\delta_{m-1,\eps}<\rho_{m-1,\eps}<\ldots<\delta_{1,\eps}<\rho_{1,\eps}=1
\end{align*}
and
  \begin{align*}
   \lim_{\eps\to 0}|u_\eps(\delta_{k,\eps})|\ (\kappa_n\epsilon)^{\frac{2k-1}{2}}
      &=D(k,m)&&  \text{for $k\in\{1,\ldots,m\}$},\\
   \lim_{\eps\to 0}\delta_{k,\eps}\ 
   (\kappa_n\epsilon)^{-\frac{2 (k n-1)}{n(n-2)}}
      &=d(k,m) &&  \text{for $k\in\{1,\ldots,m-1\}$ if $m\geq 2$},\\
   \lim_{\eps\to 0}|u_\eps'(\rho_{k,\eps})|\ (\kappa_n\epsilon)^{\frac{2 kn-3n+2}{2 (n-2)}}
      &=Z(k,m) &&  \text{for $k\in\{1,\ldots,m\}$},\\
   \lim_{\eps\to 0}\rho_{k,\eps}\ 
   (\kappa_n\epsilon)^{-\frac{2 (k-1)}{n-2}}
   &=z(k,m)&&  \text{for $k\in\{2,\ldots,m\}$  if $m\geq 2$};
  \end{align*}
in particular,
  \begin{align*}
   u_\eps(\delta_k)\sim \eps^{-\frac{1}{2}-(k-1)},\quad \delta_{k,\eps}\sim \eps^{\frac{2(n-1)}{n(n-2)}+\frac{2}{n-2}(k-1)},\quad 
   u_\eps'(\rho_{k,\eps})\sim \eps^{\frac{1}{2}-\frac{n}{n-2}(k-1)},\quad 
   \rho_{k,\eps}\sim \eps^{\frac{2}{n-2}(k-1)}.
  \end{align*} 
The coefficients $D$, $d$, $Z$, and $z$ are explicitly given by
  \begin{align*}
D(k,m)&=
(n(n-2) )^{\frac{n-2}{4}}\frac{ \Gamma (m-k+1)}{m^\frac{1}{2} \Gamma (m)},\\
d(k,m)&=(m-k)^{\frac{1}{n}} \left(\frac{m^\frac{1}{2} \Gamma (m)}{\Gamma (m-k+1)}\right)^{\frac{2}{n-2}},\\
Z(k,m)&=n^{\frac{n-2}{4}}(n-2)^{\frac{n+2}{4}}  (m-k+1)^{\frac{n-1}{n-2}} \left(\frac{\Gamma (m-k+1)}{m^\frac{1}{2} \Gamma
   (m)}\right)^{\frac{n}{n-2}},\\
z(k,m)&=(m-k+1)^{-\frac{1}{n-2}} \left(\frac{m^\frac{1}{2} \Gamma (m)}{\Gamma (m-k+1)}\right)^{\frac{2}{n-2}}.
  \end{align*}
 \end{theo}

\begin{theo}[Neumann case]\label{explicit:thm_Neumann} 
For $n\geq 3$, $m\geq 2$, and $\eps\in (0,\frac{4}{n-2})$, let $u_\varepsilon$ be a radial solution of \eqref{sceq}, \eqref{nbc} with $m-1$ interior zeros.  Let $(\delta_k)_{k=1}^m$ be the decreasing sequence of all the critical points of $u_\eps$ in $[0,1]$ and $(\rho_k)_{k=1}^{m-1}$ the decreasing sequence of all the zeros of $u_\eps$ in $[0,1]$. Then
\begin{align*}
0=\delta_{m,\eps}<\rho_{m-1,\eps}<\delta_{m-1,\eps}<\rho_{m-2,\eps}<\ldots<\rho_{1,\eps}<\delta_{1,\eps}=1
\end{align*}
and
\begin{align*}
\lim_{\eps\to 0}|u_\eps(\delta_{k,\eps})|\, 
(\kappa_n\epsilon)^{\frac{2kn-3n+2}{2n}}
&=\widetilde D(k,m)&&  \text{for $k\in\{1,\ldots,m\}$},\\
\lim_{\eps\to 0}\delta_{k,\eps}\, 
(\kappa_n\epsilon)^{-\frac{2 (k-1)}{n-2}}
&=\widetilde d(k,m)&&\text{for $k\in\{2,\ldots,m-1\}$},\\
\lim_{\eps\to 0}|u_\eps'(\rho_{k,\eps})|\,
(\kappa_n\epsilon)^{\frac{2kn-3n+4}{2 (n-2)}}
&=\widetilde Z(k,m)&&\text{for $k\in\{1,\ldots,m-1\}$},\\
\lim_{\eps\to 0}\rho_{k,\eps}\, 
(\kappa_n\epsilon)^{-\frac{2kn-2n+2}{n(n-2)}}
&=\widetilde z(k,m)&&\text{for $k\in\{1,\ldots,m-1\}$};
\end{align*}
in particular,
\begin{align*}  
u_\eps(\delta_{k,\eps})\sim \eps^{\frac{n-2}{2n}-(k-1)},\quad 
\delta_{k,\eps}\sim \eps^{\frac{2(k-1)}{n-2}}, \quad
u_\eps'(\rho_{k,\eps})\sim \eps^{\frac{n-4}{2(n-2)}-\frac{n(k-1)}{n-2}},\quad
\rho_{k,\eps}\sim \eps^{\frac{2}{n(n-2)}+\frac{2(k-1)}{n-2}}.
\end{align*}
The coefficients $\widetilde  D$, $\widetilde  d$, $\widetilde  Z$, and $\widetilde  z$ are explicitly given by
\begin{align*}
\widetilde D(k,m)&=
(n(n-2))^{\frac{n-2}{4}} 
(m-1)^{\frac{1}{2}-\frac{1}{n}}
\frac{ \Gamma (m-k+1)}{\Gamma
   (m)},\\
 \widetilde d(k,m)&= 
    (m-1)^{-\frac{1}{n}} (m-k)^{\frac{1}{n}} \left(\frac{\Gamma (m)}{\Gamma (m-k+1)}\right)^{\frac{2}{n-2}},\\
 \widetilde Z(k,m)&=
n^{\frac{n-2}{4}}  (n-2)^{\frac{n+2}{4}}  
  (m-1)^\frac{1}{2} 
  (m-k)^{-\frac{1}{n-2}} \left(\frac{\Gamma
   (m-k+1)}{\Gamma (m)}\right)^{\frac{n}{n-2}}
 ,\\
 \widetilde z(k,m)&=
 (m-1)^{-\frac{1}{n}} (m-k)^{\frac{1}{n-2}} \left(\frac{\Gamma (m)}{\Gamma (m-k+1)}\right)^{\frac{2}{n-2}}.
\end{align*}
\end{theo}

The proof of Theorems \ref{explicit:thm_Dirichlet} and \ref{explicit:thm_Neumann} is done by induction, and relies on a radial Pohozaev identity, a blow-up/normalization procedure, energy estimates, direct computations, and some ODE arguments. A more detailed discussion of this strategy can be found after Corollary \ref{ODE:coro} below.

Observe that, in the Neumann case, the behavior of $u_\eps'(\rho_{1,\eps})$ (that is, of the derivative of the solution at the largest interior zero) is particularly interesting, since its behavior changes drastically depending on the dimension (as in many situations in critical problems, see for example \cite{BN83}, dimension 4 is a threshold). Indeed we have that
\begin{align*}
&\text{for $n=3$, } \lim_{\eps\to 0}|u_\eps'(\rho_{1,\eps})|\ (\frac{\pi}{32}\epsilon)^{\frac{1}{2}} =(m-1)^{-\frac{1}{2}}3^\frac{1}{4}, \text{\ \ therefore \ \  }|u_\eps'(\rho_{1,\eps})|\to \infty\text{ as }\eps\to 0;\\
&\text{for $n=4$, } |u_\eps'(\rho_{1,\eps})|\to 4\sqrt{2}\text{ as }\eps\to 0;\\
&\text{for $n\geq 5$, } \lim_{\eps\to 0^+} \eps^\frac{4-n}{2(n-2)}=\infty, \text{\ \ therefore\ \  }u_\eps'(\rho_{1,\eps})\to 0\text{ as }\eps\to 0.
\end{align*}

Another difference between the Neumann and the Dirichlet case is the behavior of the solution at the largest critical point $\delta_{1,\eps}$, where $|u_\eps'(\delta_{1,\eps})|\to \infty$ in the Dirichlet case but $u_\eps'(\delta_{1,\eps})=u_\eps'(1)\to 0$ for Neumann b.c; in fact, the Dirichlet solution is unbounded in the nodal set that touches the boundary $\partial B$ as $\eps\to 0$, whereas the Neumann solution goes uniformly to zero as $\eps\to 0$ in this region.  Actually, our approach also yields information on the asymptotic behavior of $u_\eps(x)$ for fixed $|x|\in(0,1)$; the next result can be seen as an extension of \eqref{AP2}.

\setlength{\unitlength}{1cm}
\begin{center}
\begin{figure}[h!]
\begin{picture}(14,4)
\put(.5,4){$u_\eps(r)$}
\put(7.8,4){$u_\eps(r)$}
\put(6.8,.5){$r$}
\put(14.2,.5){$r$}
\includegraphics[width=.46\textwidth]{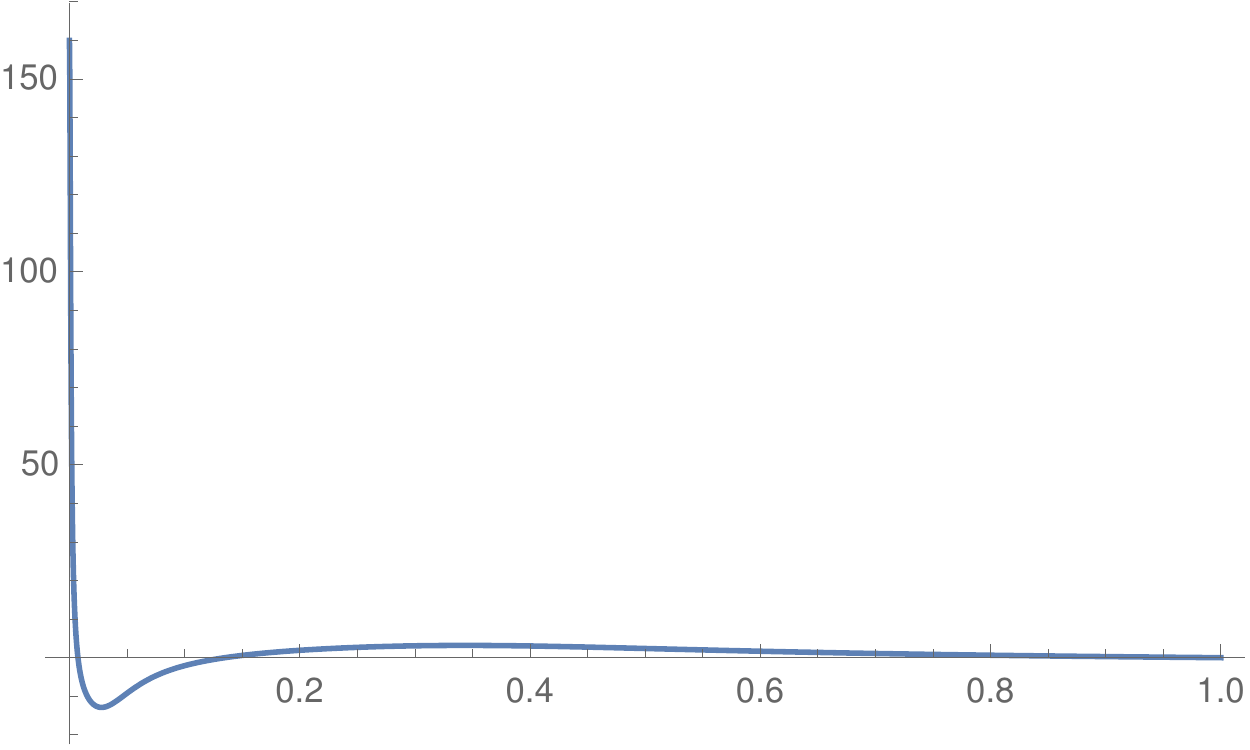}\qquad\includegraphics[width=.45\textwidth]{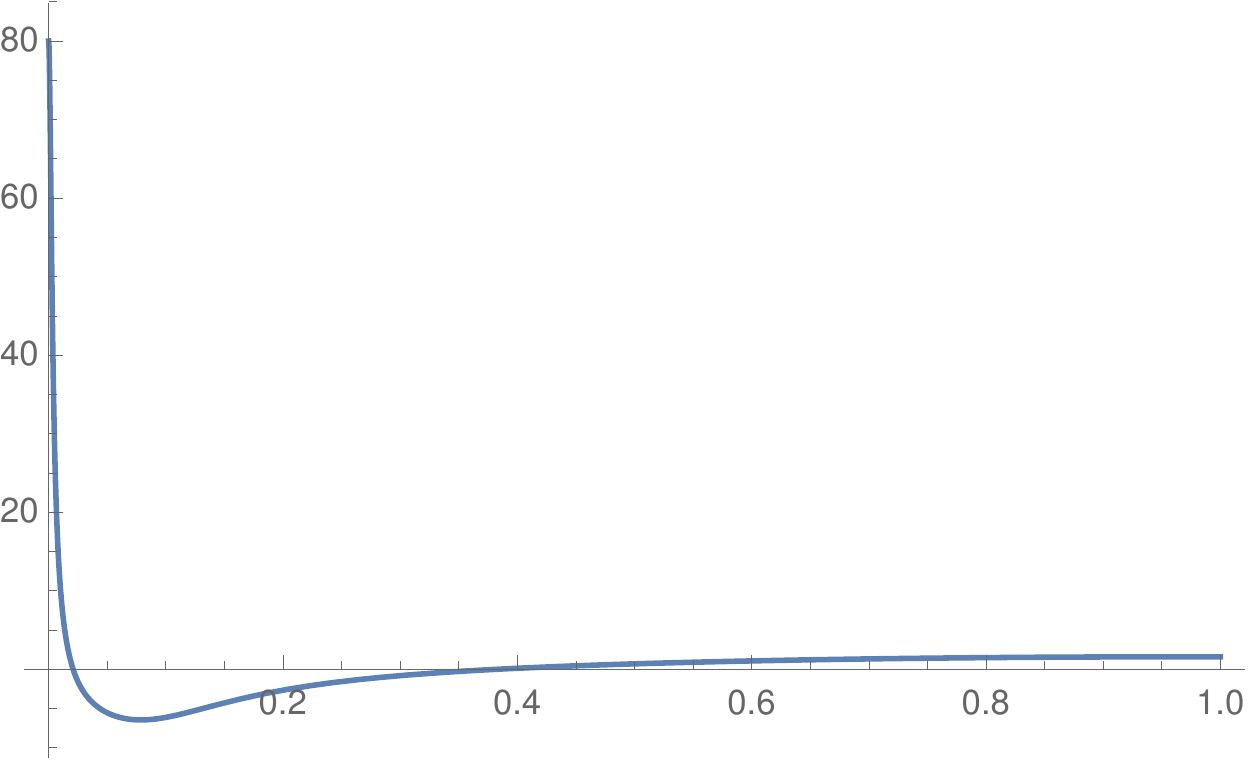} 
\end{picture}
\caption{A radial solution $u_\eps$ of \eqref{sceq} satisfying either \eqref{dbc} (on the left) or \eqref{nbc} (on the right) with two interior zeros for $\eps=1$ and $n=3$ showing concentration at the origin.}
\end{figure}
\end{center}

\begin{theo}\label{pointwise:coro}
 For $\eps\in(0,\frac{4}{n-2})$ let $u_\eps$ be a radial solution of \eqref{sceq} with $m-1$ interior zeros and fix $x\in \overline{B}\backslash\{0\}$. If $u_\eps$ satisfies Dirichlet b.c. \eqref{dbc} and $m\geq 1$,
  \begin{align*}
  |u_\eps(x)|(\kappa_n\epsilon)^{-\frac{1}{2}} = (n(n-2) )^{\frac{n-2}{4}}m^\frac{1}{2}(|x|^{2-n}-1) + o(1).
  \end{align*} 
and,  if $u_\eps$ satisfies Neumann b.c. \eqref{nbc} and $m\geq 2$,
  \begin{align*}
  &|u_\eps(x)|(\kappa_n\epsilon)^{-\frac{n-2}{2n}}=(n(n-2))^{\frac{n-2}{4}} (m-1)^{\frac{n-2}{2n}}+o(1),
   \end{align*}
 where $o(1)\to0$  as $\eps\to 0$, uniformly in compact subsets of $\overline{B}\backslash\{0\}$. In particular, $u_\eps$ converges uniformly to zero in compact subsets of $\overline{B}\backslash\{0\}$. 
\end{theo}

As mentioned above, in \cite{PW07,CD06} the authors use a Lyapunov-Schmidt reduction scheme to study the asymptotic profile of radial sign-changing solutions of \eqref{sceq}, \eqref{dbc}.  This approach uses the family of all positive solutions in $\mathcal{D}^{1,2}(\R^n)$ of the critical problem (see \cite{A76, T76})
\begin{align}\label{le}
-\Delta w=w^\frac{n+2}{n-2}\qquad \text{ in }\R^n
\end{align}
given by
\begin{align}\label{bubble}
w_{\xi,\mu}(y)=\gamma_n (1+\mu^\frac{4}{n-2}|y-\xi|^2)^\frac{2-n}{2}\mu,\qquad \xi\in\R^n,\quad \mu>0,\quad \gamma_n:=(n(n-2))^\frac{n-2}{4}.
\end{align}
Using the terminology of differential geometry, the solution $w_{\xi,\mu}$ is often referred to as \emph{single bubble}. In particular, in \cite{PW07,CD06} it is shown that a solution with exactly $(m-1)$-interior zeros has the form
\begin{align}\label{dpf}
 u_\eps(y)=\gamma_n \sum_{k=1}^m(-1)^{k+1} \Bigg(\frac{1}{1+[\alpha_k \eps^{\frac{1}{2}-k}]^\frac{4}{n-2}|y|^2}\Bigg)^\frac{n-2}{2}\alpha_k \eps^{\frac{1}{2}-k}-f_\eps(y)\eps^\frac{1}{2},\qquad y\in B,
\end{align}
where $f_\eps$ is a function which is uniformly bounded in $B$, $\alpha_k$ are some positive constants, and $\gamma_n$ is given by \eqref{bubble}.  Formula \eqref{dpf} is sometimes called a \emph{superposition of bubbles} or a \emph{tower of bubbles}.  We can use now Theorems \ref{explicit:thm_Dirichlet} and \ref{pointwise:coro} to complement \eqref{dpf}.
\begin{coro}[Dirichlet tower of bubbles]\label{dp:coro}
Given $m\geq 1$ let
 \begin{align}\label{alphak}
  \alpha_k:=\frac{\Gamma (m-k+1)}{m^\frac{1}{2} \Gamma (m)}\kappa_n^{\frac{1}{2}-k}=\frac{\Gamma (m-k+1) }{m^\frac{1}{2} \Gamma (m)}\left(\frac{(n-2) \Gamma \left(\frac{n}{2}\right)^2}{4\Gamma
   (n)}\right)^{\frac{1}{2}-k}\quad\text{ for }k\in\{0,\ldots,m\}.
 \end{align}
Then, for all sufficiently small $\eps>0$ there are only two radial solutions $u_\eps$ and $-u_\eps$ of \eqref{sceq}, \eqref{dbc} with exactly $m-1$ interior zeros in $(0,1)$ and $u_\eps$ satisfies \eqref{dpf} with $\alpha_k$ as in \eqref{alphak} and, for $K\subset\subset \overline{B}\backslash\{0\}$,
\begin{align}\label{O}
\lim_{\eps\to 0}\|f_\eps-\gamma_n\alpha_0\|_{L^\infty(K)}=
0,
\end{align}
where  $\gamma_n:=(n(n-2))^\frac{n-2}{4}.$
\end{coro}
We remark that, up to some calculations and substitutions, the constants \eqref{alphak} can also be deduced from the proofs in \cite{PW07,CD06}, which rely on different arguments than ours.  Note that the Lyapunov-Schmidt approach provides a general shape of the solution; however, it seems difficult to obtain precise information regarding rates as in Theorem \ref{explicit:thm_Dirichlet} using only \eqref{dpf}. In this regard, the Lyapunov-Schmidt scheme and our approach complement each other.

We are not aware of any result as in \cite{PW07,CD06} for the Neumann problem \eqref{sceq}, \eqref{nbc}; however, using the rates of Theorem \ref{explicit:thm_Neumann} and Corollary \ref{dp:coro} we can show the following result.

\begin{coro}[Neumann  tower of bubbles]\label{nbt:coro}
 Let $\eps\in(0,\frac{4}{n-2})$, $m\geq 2$, $n\geq 3$, and $u_\eps$ be the solution of \eqref{sceq}, 
 \eqref{nbc} with exactly $m-1$ interior zeros such that $(-1)^{m+1}u_\eps(0)>0$. For $k\in\{1,\ldots,m\}$ let
\begin{align}\label{beta}
\beta_k:=
(m-1)^{\frac{n-2}{2n}} m^{\frac{1}{2}} \kappa_n^\frac{n-1}{n}\alpha_k
=(m-1)^{\frac{n-2}{2n}}\frac{ \Gamma (m-k+1)}{\Gamma(m)}\Big( \frac{(n-2)}{4}\frac{\Gamma(\frac{n}{2})^2}{\Gamma(n)} \Big)^{\frac{n-2}{2n}-(k-1)}.
\end{align}
Then, for $y\in B$,
 \begin{align}\label{dpfn}
   u_\eps(y)=\gamma_n \sum_{k=1}^m(-1)^{k+1} \Bigg(\frac{1}{1+[\beta_{k,\eps}\, \eps^{\frac{n-2}{2n}-(k-1)}]^\frac{4}{n-2}|y|^2}\Bigg)^\frac{n-2}{2}&\beta_{k,\eps}\, \eps^{\frac{n-2}{2n}-(k-1)}+g_\eps(y)\eps^{1+\frac{n-2}{2n}},
  \end{align}
where $\lim_{\eps\to 0}\beta_{k,\eps}=\beta_k$ for $k\in\{1,\ldots,m\}$, $\gamma_n:=(n(n-2))^\frac{n-2}{4}$, and $g_\eps$ is a function which is uniformly bounded in $B$.
\end{coro}

\medskip

For our last result, we present a consequence of Theorems \ref{explicit:thm_Dirichlet} and \ref{explicit:thm_Neumann} regarding radial solutions of \eqref{ws}. These solutions have infinitely many oscillations and between two consecutive roots there is only one local maximum or minimum (see \cite[page 294]{N83}).  It is easily seen that solutions of \eqref{eq:radial} and \eqref{sceq} satisfying \eqref{dbc} or \eqref{nbc} are all connected via suitable rescalings.  As a consequence, we have the following asymptotic profiles.
\begin{coro}\label{ODE:coro}
 Let $n\in\N$, $n\geq 3$, $\eps\in (0,\frac{4}{n-2})$, and $w_\eps\in C^2([0,\infty))$ be the radial solution of 
\begin{align}\label{eq:radial}
-\Delta w_\eps=|w_\eps|^{\frac{4}{n-2}-\varepsilon}w_\eps\quad \text{in }\R^n,\qquad w_\eps(0)=1.
\end{align}
Moreover, let $(r_{i,\eps})_{i=1}^\infty$ and $(s_{i,\eps})_{i=1}^\infty$ be respectively the (divergent) increasing sequences of all zeros and critical points of $w_\eps$, such that
\[
0=s_{1,\eps}<r_{1,\eps}<s_{2,\eps}<r_{2,\eps}<\ldots < s_{i,\eps}<r_{i,\eps}<\ldots 
\]
Then, 
\begin{align*}
\lim_{\eps\to 0}r_{m,\epsilon }(\kappa_n\epsilon)^{\frac{2m-1}{n-2}}&= (n(n-2))^\frac{1}{2} m^{\frac{1}{2-n}}\  \Gamma (m)^{-\frac{2}{n-2}} &&\text{for $m\geq 1$,}\\
\lim_{\eps\to0}|w_\eps'(r_{m,\epsilon })|(\kappa_n\epsilon)^{\frac{1-m n}{n-2}}&= \left(\frac{n-2}{n}\right)^\frac{1}{2} m^{\frac{n-1}{n-2}}  \Gamma (m)^{\frac{n}{n-2}} &&\text{for $m\geq 1$,},\\
\lim_{\eps\to 0}s_{m,\eps}\, (\kappa_n\epsilon)^{\frac{2mn -3n+2}{n(n-2)}}&= (n(n-2))^\frac{1}{2} (m-1)^{\frac{1}{n}} \Gamma (m)^{-\frac{2}{n-2}} &&\text{for $m\geq 2$,},\\
\lim_{\eps\to 0}|w_\eps(s_{m,\eps})|\, (\kappa_n\epsilon)^{1-m} &=\Gamma (m) &&\text{for $m\geq 2$}.
\end{align*} 
\end{coro}

We now discuss in more detail the proofs of our main results and the relationship with Corollary \ref{ODE:coro}.  As mentioned earlier, the proof of Theorems \ref{explicit:thm_Dirichlet} and \ref{explicit:thm_Neumann} is intertwined and holds a close relationship with the solution of \eqref{eq:radial}, since we base our induction on rescalings using the sequence
\begin{align*}
 0=s_{1,\eps}<r_{1,\eps}<s_{2,\eps}<r_{2,\eps}<s_{3,\eps}<\ldots<s_{i,\eps}< r_{i,\eps}<\ldots 
\end{align*}

To explain our approach, consider first
\begin{align*}
\text{the \emph{positive} solution $v_{1,\eps}$ of the Dirichlet problem \eqref{sceq}, \eqref{dbc}.}   
\end{align*}
In this case the limiting behavior of the sequences $v_{1,\eps}(0)$ and $v'_{1,\eps}(1)$ is fully characterized by \eqref{AP1}, \eqref{AP2p}.  For each $\eps>0$ small, these limits can be seen also as an invertible nonlinear system of equations, that is,
\begin{align*}
\eps v_{1,\eps}^2(0) = D(1,1)+o(1),\qquad 
\eps^{-\frac{1}{2}} |v'_{1,\eps}(1)| &= Z(1,1)+o(1).
\end{align*}
Here $\eps$ and $\eps^{-\frac{1}{2}}$ would be the coefficients, $v_{1,\eps}(0)$ and $|v'_{1,\eps}(1)|$ the unknown variables, while $D(1,1)+o(1)$ and $Z(1,1)+o(1)$ are suitable right-hand sides (note that $o(1)$ is not explicit and may depend on $\eps$, but these terms vanish after taking the limit as $\eps\to0$). Consider now 
\begin{align*}
\text{the solution $v_{2,\eps}$ of the Neumann problem \eqref{sceq}, \eqref{nbc} with one interior zero and $v_{2,\eps}(0)>0$.} 
\end{align*}
We wish to establish the behavior of the sequences
\begin{align}\label{vars}
v_{2,\eps}(0),\quad v_{2,\eps}'(\rho_{1,\eps}),\quad  v_{2,\eps}(1),\quad \text{ and }\quad \rho_{1,\eps}, 
\end{align}
where $\rho_{1,\eps}$ is the unique zero of $v_{2,\eps}$, and $\delta_{1,\eps}=1$, $\delta_{2,\eps}=0$ are the unique critical points.  To determine \eqref{vars}, we require a suitable system of \emph{four} equations.  Since $v_{1,\eps}$ and $v_{2,\eps}$ are related (by uniqueness) via the rescaling
\begin{align*}
v_{1,\eps}(x)=\rho_{1,\eps}^{\frac{2(n-2)}{4-\eps(n-2)}}  v_{2,\eps}(x \rho_{1,\eps}),\qquad x\in(0,1),\ \eps\in(0,\frac{4}{n+2}),
\end{align*}
we can use the known information on $v_{1,\eps}$ to obtain two equations (Lemma \ref{rescaled} for $m=2$).  The other two equations (Lemma \ref{2eqs}) are obtained by some computations using the equation on $(\rho_{1,\eps},0)$ (see Lemma \ref{intlem}) together with energy estimates (Lemma \ref{Ebound}) and a normalization argument (Lemma \ref{prop:lemma}).  The resulting system of equations is nonlinear, but a solution can be found directly by substitution.

Next, consider
\begin{align*}
\text{the solution $v_{3,\eps}$ of the Dirichlet problem \eqref{sceq}, \eqref{dbc} with one interior zero and $v_{1,\eps}(0)>0$.} 
\end{align*}
In this case, the solution has two zeros, $\rho_{2,\eps}<\rho_{1,\eps}=1$ and two critical points, $\delta_{2,\eps}=0<\delta_{1,\eps}$, and we study the behavior of the sequences 
\begin{align*}
v_{3,\eps}(0), \quad
v_{3,\eps}'(\rho_{2,\eps}), \quad
v_{3,\eps}(\delta_{1,\eps}),\quad 
v_{3,\eps}'(1),\quad 
\rho_{2,\eps},\quad 
\delta_{1,\eps}.
\end{align*}
Therefore, we need a suitable system of 6 equations. As before, using a rescaling we can relate $v_{3,\eps}$ and $v_{2,\eps}$ to obtain 4 equations (Lemma \ref{Dir:res} for $m=2$); the other two (equations \eqref{A1eps} and \eqref{A2eps}) are obtained using a radial Pohozaev identity (see \eqref{poho1} in the proof of Lemma \ref{L1}) and a blow-up procedure in the set $(\delta_{1,\eps},1)$ (Lemma \ref{2nrD:lemma2}). Here the bounds obtained in \cite{DIP17} are crucial, see Theorem \ref{dip:thm}.

Finally, we can consider in a similar fashion $v_{4,\eps},$ $v_{5,\eps},$ $v_{6,\eps},\ldots$ and show Theorems \ref{explicit:thm_Dirichlet} and \ref{explicit:thm_Neumann} by induction.  In the proof that we present below in Sections \ref{N:sec} and \ref{D:sec}, we follow this inductive strategy starting from explicit formulas for the constants $d$, $D$, $z$, $Z$, $\widetilde d$, $\widetilde D$, $\widetilde z$, and $\widetilde Z$.   To deduce these formulas in the first place, we argued exactly as described above but with unknown coefficients, and we obtain recurrence identities relating $d$, $D$, $z$, $Z$, $\widetilde d$, $\widetilde D$, $\widetilde z$, and $\widetilde Z$ that unequivocally define them. These relations were then developed for $m=1,2,3,4,\ldots$ and $k=1,2,3,\ldots$, from which general formulas can be deduced. This approach requires hundreds\textemdash if not thousands\textemdash of algebraic manipulations and, to avoid any calculation mistake, we used a symbolic calculus software\footnote{Mathematica 11.1.1.0, Wolfram Research Inc., 2017.}.  Although this implementation was in itself a nontrivial computational challenge, to keep this paper short, we only present the rigorous proof by induction starting from known formulas for the coefficients.

\medskip

To close this introduction, we refer to  \cite{BDE00} 
for a broader perspective on the problem \eqref{sceq}, \eqref{dbc}; in particular, using a dynamical-system approach, the authors in \cite{BDE00} provide a full classification of the set of all positive and nodal (regular and singular) radial solutions of the equation
\begin{align*}
 -\Delta u = \lambda u + |u|^{p-1}u\quad \text{ in }B,\qquad u=0\quad \text{ on }\partial B,
\end{align*}
without any restriction on $\lambda\in\R$ and for all $p>1$.

\medskip

The paper is organized as follows. In Section \ref{pos:sec} we recall the rates for positive solutions of the Dirichlet problem (this serves as the inductive base in our argument). In Section \ref{N:sec} we study the radial Neumann problem whereas Section \ref{D:sec} is devoted to the Dirichlet problem.  The proof of Theorems  \ref{explicit:thm_Dirichlet} and \ref{explicit:thm_Neumann} can be found at the end of Section \ref{D:sec}. Section \ref{RC:sec} contains the proof of Theorem \ref{pointwise:coro} and Corollaries \ref{dp:coro}, \ref{nbt:coro}, and \ref{ODE:coro}.

\section{Positive Dirichlet solution}\label{pos:sec}
In this section we establish the induction base of our argument to show Theorems \ref{explicit:thm_Dirichlet} and \ref{explicit:thm_Neumann}.  Let $u_\eps$ be the unique positive radial solution of \eqref{sceq}, \eqref{dbc}.  The function $u_\eps$ has a unique critical point $\delta_{1,\eps}=\delta_1=0$ and a unique zero $\rho_{1,\eps}=\rho_1=1$.

\begin{theo}\label{1nD:thm}
 Let $u_\varepsilon$ be the unique positive radial solution of
 \begin{align*}
 -\Delta u_\varepsilon = u_\varepsilon^{\frac{n+2}{n-2}-\varepsilon} \quad\text{ in }B,\qquad u_\varepsilon=0\quad \text{ on }\partial B.
\end{align*}
Then
\begin{align}
 &\lim_{\varepsilon\to 0}(\kappa_n\epsilon)^{\frac{1}{2}} |u_\varepsilon(0)| = (n(n-2))^{\frac{n-2}{4}}=D(1,1), \label{1nD_aux1}\\
 &\lim_{\varepsilon\to 0} (\kappa_n\epsilon)^{-\frac{1}{2}}|u'_\varepsilon(1)|= n^{\frac{n-2}{4}}(n-2)^{\frac{n+2}{4}}=Z(1,1), \label{1nD_aux2}
\end{align}
where $\kappa_n>0$ is given by \eqref{Cchi} and $Z(1,1),$ $D(1,1)$ are as in Theorem \ref{explicit:thm_Dirichlet}.
\end{theo}
\begin{proof} This follows from \cite[Theorems A and B]{AP87} and \cite[Proposition 1]{H91}, see \eqref{AP1}, \eqref{AP2p}.
\end{proof}

\begin{remark}\label{C:rmk}
 Observe that, if $w_\eps$ is the positive solution of
 \begin{align}\label{posD}
 -\Delta w_\eps=n(n-2) w^{\frac{n+2}{n-2}-\eps} \text{ in } B,\qquad w_\eps=0 \text{ on } \partial B,
  \end{align}
then the constants in \eqref{1nD_aux1}, \eqref{1nD_aux2} are simpler. Indeed, let $\kappa_n$ as in \eqref{Cchi} and use the rescaling $w_\varepsilon(x)=(n(n-2))^{\frac{n-2}{\varepsilon(n-2)-4}} u_\varepsilon(x)$; then
 \begin{align*}
 &\lim_{\varepsilon\to 0}(\kappa_n\epsilon)^\frac{1}{2}w_\varepsilon(0) = 1\qquad \text{ and }\qquad 
 \lim_{\varepsilon\to 0} (\kappa_n\epsilon)^{-\frac{1}{2}}|w'_\varepsilon(1)|=n-2.
  \end{align*}
\end{remark}

\section{The radial Neumann problem}\label{N:sec}
As explained in the introduction, the proofs of Theorems \ref{explicit:thm_Dirichlet} and \ref{explicit:thm_Neumann} are intertwined and performed together by induction. In the previous section we proved the starting point of the induction procedure: Theorem \ref{explicit:thm_Dirichlet} for $m=1$. The purpose of this section is to prove the following.

\begin{prop}\label{prop1}
  Let $m\geq 2$. If Theorem \ref{explicit:thm_Dirichlet} (Dirichlet case) holds for radial solutions with $m-2$ interior zeros, then Theorem \ref{explicit:thm_Neumann} (Neumann case) holds for radial solutions with $m-1$ interior zeros.
\end{prop}

Let $n\geq 3$, $m\geq 2$, $\eps\in (0,\frac{4}{n-2})$, and let $u_\eps$ be a solution of \eqref{sceq}, \eqref{nbc} with $m-1$ interior zeros.  Let $1= \delta_{1,\eps}>\delta_{2,\eps}>\ldots >\delta_{m,\eps}=0$ be the decreasing sequence of all the critical points of $u_\eps$ in $[0,1]$ and $1> \rho_{1,\eps}>\ldots>\rho_{m-1,\eps}>0$ the decreasing sequence of all the zeros of $u_\eps$ in $[0,1]$.   The constants $d$, $D$, $z$, $Z$, $\widetilde d$, $\widetilde D$, $\widetilde z$, and $\widetilde Z$ are explicit constants given by Theorems \ref{explicit:thm_Dirichlet} and \ref{explicit:thm_Neumann}. The constant $\kappa_n$ is given in \eqref{Cchi}. Observe that, in virtue of the Neumann boundary conditions and the uniqueness of the Cauchy problem for radial solutions of \eqref{sceq} (or by Hopf's Lemma), we have
\begin{align}\label{unot0}
 u(1)\neq 0.
\end{align}


\begin{lemma}\label{rescaled}
  Let $m\geq 2$. If Theorem \ref{explicit:thm_Dirichlet} (Dirichlet case) holds for radial solutions with $m-2$ interior zeros, then
\begin{align}
\lim_{\eps\to 0}\rho_{1,\eps}^{\frac{2(n-2)}{4-\eps(n-2)}} 
  |u_\eps(\delta_{k,\eps})|\ (\kappa_n\epsilon)^{\frac{2k-3}{2}}
      &=D(k-1,m-1),&&  \text{$k\in\{2,\ldots,m\}$},\label{N21}\\
   \lim_{\eps\to 0}\delta_{k,\eps}\rho_{1,\eps}^{-1}\ 
   (\kappa_n\epsilon)^{-\frac{2 (k-1) n-2}{n(n-2)}}
      &=d(k-1,m-1),&&  \text{$k\in\{2,\ldots,m-1\}$, $m\geq 3$}\label{N2},\\
        \lim_{\eps\to 0}\rho_{1,\eps}^{\frac{2(n-2)}{4-\eps(n-2)}+1}|u_\eps'(\rho_{k,\eps})|\ (\kappa_n\epsilon)^{\frac{2 kn-3n+2}{2 (n-2)}}
      &=Z(k,m-1),&&  \text{$k\in\{1,\ldots,m-1\}$},\label{R}\\
   \lim_{\eps\to 0}\rho_{k,\eps}\rho_{1,\eps}^{-1}\ 
   (\kappa_n\epsilon)^{-\frac{2 (k-1)}{n-2}}
   &=z(k,m-1),&&  \text{$k\in\{2,\ldots,m-1\}$, $m\geq 3$},\label{R1}
\end{align}
where $D$, $d,$ $R,$ and $r$ are as in Theorem \ref{explicit:thm_Dirichlet}. 
\end{lemma}
\begin{proof}
To ease notation we omit the $\eps$ dependency of $u_\eps, \delta_{k,\eps},\rho_{k,\eps}$. Observe that $w:B\to\R$ given by $w(x):=\rho_1^{\frac{2(n-2)}{4-\eps(n-2)}}  u(x \rho_1)$ is a Dirichlet solution with $m-2$ interior zeros. If  $1> \widetilde \delta_1>\ldots >\widetilde \delta_{m-1}=0$ and $1= \widetilde \rho_1>\ldots>\widetilde \rho_{m-1}>0$ are the critical points and roots of $w$ then, since Theorem \ref{explicit:thm_Dirichlet} holds for solutions with $m-2$ interior zeros, we have that
\begin{align*}
   \lim_{\eps\to 0}|w(\widetilde\delta_k)|\ (\kappa_n\epsilon)^{\frac{2k-1}{2}}
      &=D(k,m-1)&&  \text{for $k\in\{1,\ldots,m-1\}$},\\
   \lim_{\eps\to 0}\widetilde\delta_{k}\ 
   (\kappa_n\epsilon)^{-\frac{2(k n-1)}{n(n-2)}}
      &=d(k,m-1) &&  \text{for $k\in\{1,\ldots,m-2\}$ if $m\geq 3$},\\
   \lim_{\eps\to 0}|w'(\widetilde\rho_k)|\ (\kappa_n\epsilon)^{\frac{2 kn-3n+2}{2 (n-2)}}
      &=Z(k,m-1) &&  \text{for $k\in\{1,\ldots,m-1\}$},\\
   \lim_{\eps\to 0}\widetilde\rho_{k}\ 
   (\kappa_n\epsilon)^{-\frac{2 (k-1)}{n-2}}
   &=z(k,m-1)&&  \text{for $k\in\{2,\ldots,m-1\}$  if $m\geq 3$},
  \end{align*}
The result now follows from the definition of $w$ and the fact that $\widetilde \delta_{k}\rho_1=\delta_{k+1}$ and $\widetilde \rho_k \rho_1=\rho_k$ for $k\in\{1,\ldots,m-1\}$.
\end{proof}

\begin{lemma}\label{intlem}
 It holds that 
 \begin{align*}
 |u_\eps(1)|^{\frac{4}{n-2}-\eps}=\frac{(n-2)\rho_{1,\eps}^{n-2}}{\int_{\rho_{1,\eps}}^1 \Big|\frac{u_\eps(s)}{u_\eps(1)}\Big|^{\frac{4}{n-2}-\eps}\frac{u_\eps(s)}{u_\eps(1)}s( s^{n-2} -  \rho_{1,\eps}^{n-2})\ ds}.
 \end{align*}
\end{lemma}
\begin{proof}
To ease notation we omit the $\eps$ dependency on $u_\eps$ and $\rho_{1,\eps}$.  Assume without loss of generality that $u(1)>0$ (recall \eqref{unot0}).  Using that 
\begin{align}\label{radeq}
 -(u' r^{n-1})' = |u|^{\frac{4}{n-2}-\eps}u \ r^{n-1}\quad \text{ in }(0,1),\qquad u'(0)=u'(1)=0,
\end{align}
we have
\begin{align}\label{1ibyp}
\int_{r}^1 |u(s)|^{\frac{4}{n-2}-\eps}u(s) \ s^{n-1}\ ds =\int_{r}^1 -(u'(s) s^{n-1})'\ ds=u'(r) r^{n-1}\qquad \text{ for $r\in(0,1)$}.
\end{align}
Then, since $u>0$ in $(\rho_1,1]$,
\begin{align*}
u(1)= \int_{\rho_1}^1 u'(t)\ dt
=\int_{\rho_1}^1 t^{1-n}\int_{t}^1 u(s)^{\frac{n+2}{n-2}-\eps} \ s^{n-1}\ ds\ dt>0
.
\end{align*}
Let $f(t)=\frac{t^{2-n}}{2-n}$ and $g(t):=\int_{t}^1 u(s)^{\frac{n+2}{n-2}-\eps} \ s^{n-1}\ ds$. Then  $f'(t):=t^{1-n}$, $g'(t)=-u(t)^{\frac{n+2}{n-2}-\eps} t^{n-1}$ and $g(1)=0$, and since 
$\int_{\rho_1}^1 f'g= fg\mid_{\rho_1}^1-\int_{\rho_1}^1 fg'= -f(\rho_1)g(\rho_1)-\int_{\rho_1}^1 fg'$ we obtain
\begin{align}
u(1)&=-\frac{\rho_1^{2-n}}{2-n}\int_{\rho_1}^1 u(s)^{\frac{n+2}{n-2}-\eps} \ s^{n-1}\ ds - \int_{\rho_1}^1 \frac{t^{2-n}}{2-n}(-u(t)^{\frac{n+2}{n-2}-\eps} \ t^{n-1})\ dt\label{di}\\
&=\frac{1}{n-2}\Big(\rho_1^{2-n}\int_{\rho_1}^1 u(s)^{\frac{n+2}{n-2}-\eps} \ s^{n-1}\ ds - \int_{\rho_1}^1 tu(t)^{\frac{n+2}{n-2}-\eps}\ dt\Big)\nonumber\\
&=\frac{u(1)^{\frac{n+2}{n-2}-\eps}}{n-2}\Big(\rho_1^{2-n}\int_{\rho_1}^1 \Big(\frac{u(s)}{u(1)}\Big)^{\frac{n+2}{n-2}-\eps}s( s^{n-2} -  \rho_1^{n-2})\ ds\Big),\nonumber
\end{align}
and the claim follows. 
\end{proof}

Define, as usual, $2^*=\frac{2n}{n-2}$ the critical Sobolev exponent.

\begin{lemma}\label{Ebound}
There is $C_1>0$ independent of $\eps$ such that
\begin{align*}
\|u_\eps\|_{L^{2^*-\eps}(B)}+\|\nabla u_\eps\|_{L^{2}(B)}< C_1\qquad \text{ for all $\eps>0$ sufficiently small}.
\end{align*}
\end{lemma}
\begin{proof}
Let $v_\eps$ be a solution of \eqref{sceq} with Dirichlet boundary conditions \eqref{dbc} and $m-1$ interior zeros and let $1> \widehat\delta_{1,\eps}>\widehat\delta_{2,\eps}>\ldots >\widehat\delta_{m,\eps}=0$ be the decreasing sequence of all the critical points of $v_\eps$ in $[0,1]$.  Then $u_\eps(x)=\widehat\delta_{1,\eps}^{\frac{2(n-2)}{4-\eps(n-2)}}  v_\eps(x \widehat\delta_{1,\eps})$ for $x\in B$ and for all $\eps>0$ small, by uniqueness.

By \cite[Proposition 3.1]{DIP17} we know that $\int_B |v_\eps(y)|^{2^*-\eps} dy \to m S_n^\frac{n}{2}$, where $S_N$ is the best Sobolev constant. Then,  
\begin{align*}
 \int_B |u_\eps|^{2^*-\eps} 
& = \int_B \widehat\delta_{1,\eps}^{\frac{2(n-2)}{4-\eps(n-2)}(2^*-\eps)} |v_\eps(x \widehat\delta_{1,\eps})|^{2^*-\eps}\\
& = \int_{B_{\widehat\delta_{1,\eps}}} \widehat\delta_{1,\eps}^{\frac{\epsilon  (n-2)^2}{4-\epsilon  (n-2)}} |v_\eps(y)|^{2^*-\eps}
 \leq \int_B |v_\eps(y)|^{2^*-\eps}\leq m S_n^\frac{n}{2}+1
\end{align*}
for $\eps$ small enough. Finally, since $u_\eps$ is a solution of \eqref{sceq}, \eqref{nbc} we have that $\|\nabla u_\eps\|_{L^{2}(B)}^2=\|u_\eps\|^{2^*-\eps}_{L^{2^*-\eps}}(B)$ and the claim follows. 
\end{proof}

\begin{lemma}\label{bds}
 There is $M>1$ independent of $\eps$ such that
 \begin{align*}
  \text{$|u_\eps(1)|<M$\qquad and\qquad $\rho_{1,\eps}<1-M^{-1}$\qquad for all $\eps>0$ sufficiently small.}
 \end{align*}
 \end{lemma}
\begin{proof}
 By Lemma \ref{Ebound}, there is $C_1>0$ independent of $\eps$ such that $\|u_\eps\|_{L^{2^*-\eps}(B)}+\|\nabla u_\eps\|_{L^{2}(B)}< C_1$ for all $\eps$ sufficiently small. In particular, the $L^{2^*-\eps}$ bound implies the existence of $\bar \eps,C_2>0$ such that, for every $\eps\in (0,\bar \eps)$, there is $x_\eps\in (\frac{1}{4},\frac{1}{2})$ such that $|u_\eps(x_\eps)|<C_2$. But then 
\begin{align*}
 |u_\eps(x)-u_\eps(x_\eps)|&=\left|\int_{x_\eps}^{x} u_\eps'(t)t^\frac{n-1}{2}t^\frac{1-n}{2}\ dt\right| \leq \left(\int_{0}^{1} |u_\eps'(t)|^2t^{n-1}\ dt\right)^\frac{1}{2} \left(\int_{\frac{1}{4}}^{1} t^{1-n}\ dt \right)^\frac{1}{2}\\
 			&= \frac{1}{|\partial B|^\frac{1}{2}} \| \nabla u_\eps\|_{L^2(B)}\left(\int_{\frac{1}{4}}^{1} t^{1-n}\ dt \right)^\frac{1}{2}\leq C_3
\end{align*}
for some $C_3>0$ independent of $\eps$ and for all $x\in (x_\eps,1]$.  In particular, $|u_\eps(1)|\leq C_4$ for some $C_4>0$ and for all $\eps$ sufficiently small.

Finally, assume by contradiction that $\rho_{1,\eps}\to 1$ (up to a subsequence). Then, using that $\max_{(\rho_{1,\eps},1)}|u_\eps|=|u_\eps(1)|$ we have
\[
\left| \int_{\rho_{1,\eps}}^1 \Big|\frac{u_\eps(s)}{u_\eps(1)}\Big|^{\frac{4}{n-2}-\eps}\frac{u_\eps(s)}{u_\eps(1)}s( s^{n-2} -  \rho_{1,\eps}^{n-2})\ ds \right| \leq \int_{\rho_{1,\eps}}^1 s(s^{n-2}-\rho_{1,\eps}^{n-2})\ ds \to 0
\] 
and from Lemma \ref{intlem} we get $|u_\eps(1)|\to \infty$, a contradiction. Therefore $\rho_{1,\eps}$ is bounded away from 1.
\end{proof}

\begin{lemma}\label{prop:lemma}
 Let $m\geq 2$. If Theorem \ref{explicit:thm_Dirichlet} (Dirichlet case) holds for radial solutions with $m-2$ interior zeros, then $\rho_{1,\eps}\to 0$ and $|\frac{u_\varepsilon}{u_\varepsilon(1)}|\to 1$ in $[a,1]$ uniformly as $\varepsilon\to 0$ for any $a>0$.
\end{lemma}
\begin{proof}
We show that every subsequence of $u_\eps$ has a subsequence for which the claim holds.  Assume without loss of generality that $u_\eps(1)>0$ and let $w_\eps(x):=\frac{u_\eps(x)}{u_\eps(1)}$.  Then $w_\eps$ solves
\begin{align*}
 -(r^{n-1}w_\eps')' =u_\eps(1)^{\frac{4}{n-2}-\eps} w_\eps^{\frac{n+2}{n-2}-\eps} r^{n-1}\quad  \text{ in }(\rho_{1,\eps},1),\quad w_\eps(\rho_{1,\eps})=w_\eps'(1)=0,\quad w_\eps(1)=1.
\end{align*}

Passing to a subsequence, the limit $\rho^*:=\lim_{\eps\to 0}\rho_{1,\eps} \geq 0$ exists. Let us prove that $\rho^*=0$.
Observe that $w_\eps$ is positive and monotone increasing in $(\rho_{1,\eps},1)$, thus $0\leq w_\eps\leq 1$ in $[\rho_{1,\eps},1]$ (by Lemma \ref{bds} we know that $\rho_{1,\eps}$ is bounded away from 1). Integrating the equation we obtain
\begin{align}\label{eq:weps}
w_\eps'(r) r^{n-1}= u_\eps(1)^{\frac{4}{n-2}-\eps} \int_r^1 w_\eps(s)^{\frac{n+2}{n-2}-\eps} s^{n-1}\, ds,\qquad r\in (\rho_{1,\eps},1),
\end{align}
and so $w_\eps'$ is uniformly bounded in $[a,1]$ for every $a>\rho^*$ (and $\eps$ sufficiently small). Since $w_\eps''(r)=-\frac{n-1}{r}w_\eps'(r) -u_\eps(1)^{\frac{4}{n-2}-\eps}w_\eps(r)^{\frac{n+2}{n-2}-\eps} r^{n-1}$, then also $w_\eps''$ is uniformly bounded in $[a,1]$ for every $a>\rho^*$.
By Rellich-Kondrachov compactness theorem, we may define $w_0(x):=\lim_{\eps\to 0}w_{\eps}(x)$ for $x\in(\rho^*,1)$, and $w_\eps\to w_0$ strongly in $C^{1,\alpha}([a,1])$ for every $a>\rho^*$.  Using again the equation, the convergence is actually in $C^{2,\alpha}([a,1])$. Then 
\begin{align}\label{eq:w_0}
-(r^{n-1}w_0')' =M^{\frac{4}{n-2}}w_0^{\frac{n+2}{n-2}} r^{n-1}\quad  \text{ in }(\rho^*,1),\qquad w'_0(1)=0\quad \text{ and }\quad w_0(1)=1,
\end{align}
where $M:=\lim_{\eps\to 0}u_\eps(1)\ge0$ (up to a subsequence), and in particular $w_0\not\equiv 0$. If $\rho^*>0$, then, by \eqref{eq:weps} and dominated convergence,
\begin{align*}
0\leq w_0(\rho^*+\eta)&=\lim_{\eps\to 0}\int^{\rho^*+\eta}_{\rho_{1,\eps}} r^{1-n} u_\eps(1)^{\frac{4}{n-2}-\eps} \int_r^1 w_\eps(s)^{\frac{n+2}{n-2}-\eps} s^{n-1}\, ds\, dr\\
&\leq \int^{\rho^*+\eta}_{\rho^*} r^{1-n} M^{\frac{4}{n-2}} \, dr
\leq (\rho^*)^{1-n}M^{\frac{4}{n-2}}\eta
\end{align*}
for all $\eta>0$ small, and therefore $w_0(\rho^*)=\lim_{\eta\to0}w_0(\rho^*+\eta)=0$. Then, by \eqref{R} for $k=1$, we have
\begin{align*}
 Z(1,m-1)=\lim_{\eps\to 0}\rho_{1,\eps}^{\frac{2(n-2)}{4-\eps(n-2)}+1}|u_\eps'(\rho_{1,\eps})|\ (\kappa_n\epsilon)^{-\frac{1}{2}}
 =\kappa_n^{-\frac{1}{2}}(\rho^*)^{\frac{n-2}{2}+1}\lim_{\eps\to 0}|u_\eps'(\rho_{1,\eps})|\ \epsilon^{-\frac{1}{2}}.
 \end{align*}
Thus, $u_\eps'(\rho_{1,\eps})\to 0$ as $\eps\to 0$. If $M=0$ then, 
by \eqref{eq:w_0}, we have that $w_0(r)=\frac A{r^{n-2}}+B$ and since $w_0'(1)=0$ and $w_0(1)=1$ we get $w_0(r)=1$ for $r\in(0,1)$, which contradicts $w_0(\rho^*)=0$. If $M>0$ then since $w_\eps(x):=\frac{u_\eps(x)}{u_\eps(1)}$ we obtain $w'_0(\rho^*)=0$, which again violates the uniqueness of the Cauchy problem.  Therefore $\rho_{1,\eps}\to 0$ as $\eps\to 0$.

Since $w_0\not \equiv 0$, $0\leq w_\eps\leq 1$ in $[\rho_{1,\eps},1]$, and $\rho_{1,\eps}\to 0$, we have by dominated convergence that
\[
\lim_{\eps\to 0}\int_{\rho_{1,\eps}}^1 w_\eps^{\frac{n+2}{n-2}-\eps} s(s^{n-2}-\rho_{1,\eps}^{n-2})\ dx= \int_0^1 w_0^{\frac{n+2}{n-2}-\eps} s^{n-1}\ ds >0,
\]
but then, by Lemma \ref{intlem} and the fact that $\rho_{1,\eps}\to 0$, we conclude that $M=\lim_{\eps\to 0} u_\eps(1)=0$ 
and, arguing as before using \eqref{eq:w_0}, we obtain that $w_0(r)=1$ for $r\in(0,1]$.
\end{proof}

\begin{lemma}\label{2eqs}
Recall that $\delta_{1,\eps}=1$.  It holds that
\begin{equation}\label{eqs}
\begin{aligned}
&\lim_{\eps\to 0}|u_\eps(\delta_{1,\eps})|^{\frac{4}{n-2}-\varepsilon} \rho_{1,\eps}^{2-n} = n(n-2)\quad \text{ and }\quad 
&\lim_{\eps\to 0}|u_\eps'(\rho_{1,\eps})|\rho_{1,\eps}^{n-1} |u_\eps(\delta_{1,\eps})|^{\varepsilon-\frac{n+2}{n-2}}=n^{-1}.
\end{aligned}
\end{equation}
\end{lemma}
\begin{proof} Without loss of generality assume that $u_\eps(\delta_{1,\eps})=u_\eps(1)<0$ (see \eqref{unot0}). The first identity in \eqref{eqs} follows from Lemmas \ref{intlem}, \ref{prop:lemma} and Lebesgue's dominated convergence, since
\begin{align}\label{co}
&\lim_{\varepsilon\to 0} \int_{\rho_{1,\eps}}^1 \frac{|u_\varepsilon(s)|^{\frac{n+2}{n-2}-\eps}}{|u_\varepsilon(1)|^{\frac{n+2}{n-2}-\eps}}( s^{n-1} -  s\rho_{1,\eps}^{n-2})\ ds = \int_0^1 s^{n-1}= \frac{1}{n}.
\end{align}
The second identity follows similarly evaluating \eqref{1ibyp} at $\rho_{1,\eps}$ and using a similar normalization.
\end{proof}

\begin{proof}[Proof of Proposition \ref{prop1}] In the following, as usual, $o(1)$ is a function of $\eps$ such that $\lim_{\eps\to 0}o(1)=0$. Using \eqref{eqs} and \eqref{R} for $k=1$, we have
\begin{equation}\label{simeq}
\begin{aligned}
|u_\eps'(\rho_{1,\eps})|\rho_{1,\eps}^{n-1} |u(\delta_{1,\eps})|^{\varepsilon-\frac{n+2}{n-2}} &= n^{-1}+o(1),\\
|u_\eps(\delta_{1,\eps})|^{\frac{4}{n-2}-\varepsilon} \rho_{1,\eps}^{2-n} &= n(n-2)+o(1),\\
\rho_{1,\eps}^{\frac{2(n-2)}{4-\eps(n-2)}+1}|u_\eps'(\rho_{1,\eps})|\ (\kappa_n\epsilon)^{-\frac{1}{2}} &= Z(1,m-1)+o(1).
\end{aligned}
\end{equation}
For each $\eps$ these three equations establish the values of the unknowns $\rho_1$, $|u'(\rho_{1})|$, and $|u(\delta_1)|$, which we can determine by direct substitution. Namely, from the first two equations we obtain
\begin{align}\label{simeq2}
 |u_\eps'(\rho_{1,\eps})|&= \rho_{1,\eps}^{1-n} |u(\delta_{1,\eps})|^{\frac{n+2}{n-2}-\eps}(n^{-1}+o(1)),\\
 |u_\eps(\delta_{1,\eps})|&=\left((n(n-2)+o(1)) \rho_{1,\eps}^{n-2}\right)^{\frac{n-2}{4-\eps( n-2)}},
\end{align}
which we substitute in the last equation in \eqref{simeq} to have
\begin{align}
\label{m0}\begin{aligned}
 \rho_{1,\eps}^{\frac{2 (n-2)}{4-\eps (n-2)}-n+2} \left(\left((n(n-2)+o(1))\rho_{1,\eps}^{n-2}\right)^{\frac{n-2}{4-\eps( n-2)}}\right)^{\frac{n+2}{n-2}-\eps}&(n^{-1}+o(1))\\
 &=\sqrt{\kappa_n\epsilon } (Z(1,m-1)+o(1)).
\end{aligned}\end{align}
A direct computation leads to
\[
\rho_{1,\eps}=(\kappa_n \eps)^\frac{4-\eps(n-2)}{2n(n-2)} (n(n-2)+o(1))^{-\frac{n+2-\eps(n-2)}{n(n-2)}} [nZ(1,m-1)+o(1)]^\frac{4-\eps(n-2)}{n(n-2)}.
\]
Since 
\begin{align}\label{sim}
 \lim_{\eps\to 0}\eps^\frac{a+b\eps}{c+d\eps} \eps^{-\frac{a}{c}}=1\quad \text{ for all $a,b,c,d\in \R$ with $c\neq 0$},
\end{align}
then
\begin{align}\label{rho1}
\lim_{\eps\to 0}\rho_{1,\eps} (\kappa_n\epsilon)^{-\frac{2}{n(n-2)}} = (n-2)^{-\frac{n+2}{n(n-2)}} n^{-\frac{1}{n}} Z(1,m-1)^{\frac{4}{n(n-2)}}=(m-1)^{\frac{2}{n(n-2)}}=\widetilde z(1,m).
\end{align}
Now, using the exact profile of $\rho_{1,\eps}$ and \eqref{sim},
we can determine  $|u_\eps'(\rho_{1,\eps})|$ and $|u_\eps(\delta_{1,\eps})|$, that is,
\begin{align*}
\lim_{\eps\to 0}(\kappa_n\epsilon)^{\frac{2-n}{2n}} |u_\eps(\delta_{1,\eps})|
&=((n-2) n)^{\frac{n-2}{4}} \widetilde z(1,m)^{\frac{1}{4} (n-2)^2}\\
&= (n(n-2))^\frac{n-2}{4}(m-1)^\frac{n-2}{2n}= \widetilde D(1,m),\\
\lim_{\eps\to 0}(\kappa_n\epsilon)^{\frac{4-n}{2(n-2)}} |u_\eps'(\rho_{1,\eps})|
&=(n-2)^{\frac{n+2}{4}} n^{\frac{n-2}{4}} \widetilde z(1,m)^{\frac{n(n-4)}{4} }\\
&=(n-2)^\frac{n+2}{4} n^\frac{n-2}{4}(m-1)^\frac{n-4}{2(n-2)}
=\widetilde Z(1,m).
\end{align*}
We now use the profile $\rho_{1,\eps}$, \eqref{N21}--\eqref{R1}, and the fact that $\Gamma(k)=(k-1)!$ for $k\in \N$ to determine the other unknowns, namely, 
\begin{align*}
&\lim_{\eps\to 0}|u_\eps(\delta_{k,\eps})|(\kappa_n\epsilon)^{\frac{2kn-3n+2}{2n}}
 =D(k-1,m-1) \widetilde z(1,m)^{\frac{2-n}{2}} 
 =\widetilde D(k,m)\ \ \text{ for }k\in\{2,\ldots,m\},\\
 &\lim_{\eps\to 0} \delta_{k,\eps} (\kappa_n\epsilon)^{-\frac{2(k-1)}{n-2}}
=\widetilde z(1,m) d(k-1,m-1)
=\widetilde d(k,m)\ \ \text{ for }k\in\{2,\ldots,m-1\}\text{  if $m\geq 3$},\\
& \lim_{\eps\to 0}|u_\eps'(\rho_{k,\eps})|(\kappa_n\epsilon)^{\frac{2kn-3n+4}{2(n-2)}}
 =Z(k,m-1) \widetilde z(1,m)^{-\frac{n}{2}}
 = \widetilde Z(k,m)\ \ \text{ for }k\in\{1,\ldots,m-1\},\\
& \lim_{\eps\to 0}\rho_{k,\eps} (\kappa_n\epsilon)^{-\frac{2nk-2n+2 }{(n-2) n}}
 =\widetilde z(1,m) z(k,m-1)
 = \widetilde z(k,m)\ \ \text{ for }k\in\{2,\ldots,m-1\} \text{  if $m\geq 3$}.
\end{align*}
This ends the proof. 
\end{proof}

\section{The radial Dirichlet problem}\label{D:sec}

In this section we prove the following result.  

\begin{prop}\label{prop2}
 Let $m\geq 2$. If Theorem \ref{explicit:thm_Neumann} (Neumann case) holds for radial solutions with $m-1$ interior zeros, then Theorem \ref{explicit:thm_Dirichlet} (Dirichlet case) holds for radial solutions with $m-1$ interior zeros.
\end{prop}

This proposition is the final ingredient in the proof of Theorems \ref{explicit:thm_Dirichlet} and \ref{explicit:thm_Neumann}, which can be found at the end of this section.

Let $n\geq 3$, $m\geq 2$, $\eps\in (0,\frac{4}{n-2})$, and let $u_\eps$ be a solution of \eqref{sceq}, \eqref{dbc} with $m-1$ interior zeros. Let $1> \delta_{1,\eps}>\delta_{2,\eps}>\ldots >\delta_{m,\eps}=0$ be the decreasing sequence of all the critical points of $u_\eps$ in $[0,1]$ and $1= \rho_{1,\eps}>\ldots>\rho_{m,\eps}>0$ the decreasing sequence of all the zeros of $u_\eps$ in $[0,1]$.  The constants $d$, $D$, $z$, $Z$, $\widetilde d$, $\widetilde D$, $\widetilde z$, and $\widetilde Z$ are explicit constants given by Theorems \ref{explicit:thm_Dirichlet} and \ref{explicit:thm_Neumann}.  The constant $\kappa_n$ is given in \eqref{Cchi}.

\medskip

\begin{lemma}\label{Dir:res}
 Let $m\geq 2$. If Theorem \ref{explicit:thm_Neumann} (Neumann case) holds for radial solutions with $m-1$ interior zeros, then 
 \begin{align}
 \lim_{\eps\to 0}\delta_{1,\eps}^{\frac{2(n-2)}{4-\eps(n-2)}}|u_\eps(\delta_{k,\eps})|\ 
(\kappa_n\epsilon)^{\frac{2kn-3n+2}{2n}}
&=\widetilde D(k,m)&&  \text{for $k\in\{1,\ldots,m\}$},\label{A0}\\
\lim_{\eps\to 0}\delta_{1,\eps}^{-1}\delta_{k,\eps}\ 
(\kappa_n\epsilon)^{-\frac{2 (k-1)}{n-2}}
&=\widetilde d(k,m)&&\text{for $k\in\{2,\ldots,m-1\}$  if  $m\geq 3$},\label{A01}\\
 \lim_{\eps\to 0}\delta_{1,\eps}^{\frac{2(n-2)}{4-\eps(n-2)}+1}|u_\eps'(\rho_{k,\eps})|\
(\kappa_n\epsilon)^{\frac{2(k-1)n-3n+4}{2 (n-2)}}
&=\widetilde Z(k-1,m)&&\text{for $k\in\{2,\ldots,m\}$},\label{A02}\\ 
\lim_{\eps\to 0} \delta_{1,\eps}^{-1}\rho_{k,\eps}\ 
(\kappa_n\epsilon)^{\frac{-2(k-1)n+2n-2}{n(n-2)}}
&=\widetilde z(k-1,m)&&\text{for $k\in\{2,\ldots,m\}$}\label{A03}. 
 \end{align}
\end{lemma}
\begin{proof}
To ease notation we omit the $\eps$ subindices of $u_\eps$, $\delta_{k,\eps}$ and $\rho_{k,\eps}$. Observe that $w:B\to\R$ given by $w(x):=\delta_1^{\frac{2(n-2)}{4-\eps(n-2)}}  u(x \delta_1) $
is a Neumann solution with $m-1$ interior zeros. Let $1= \widehat \delta_1>\widehat\delta_2>\ldots >\widehat\delta_{m}=0$ be the decreasing sequence of all the critical points of $w$ in $[0,1]$ and $1> \widehat\rho_1>\widehat\rho_2>\ldots>\widehat\rho_{m-1}>0$ the decreasing sequence of all the zeros of $w$ in $[0,1]$. Then, by Theorem \ref{explicit:thm_Neumann},
\begin{align*}
\lim_{\eps\to 0}|w(\widehat\delta_k)|\ 
(\kappa_n\epsilon)^{\frac{2kn-3n+2}{2n}}
&=\widetilde D(k,m)&&  \text{for $k\in\{1,\ldots,m\}$},\\
\lim_{\eps\to 0}\widehat\delta_{k}\ 
(\kappa_n\epsilon)^{-\frac{2 (k-1)}{n-2}}
&=\widetilde d(k,m)&&\text{for $k\in\{2,\ldots,m-1\}$ if $m\geq 3$},\\
\lim_{\eps\to 0}|w'(\widehat\rho_k)|\
(\kappa_n\epsilon)^{\frac{2kn-3n+4}{2 (n-2)}}
&=\widetilde Z(k,m)&&\text{for $k\in\{1,\ldots,m-1\}$},\\
\lim_{\eps\to 0} \widehat \rho_k\ 
(\kappa_n\epsilon)^{-\frac{2kn-2n+2}{n(n-2)}}
&=\widetilde z(k,m)&&\text{for $k\in\{1,\ldots,m-1\}$}.
\end{align*}
The claim follows since $\widehat \rho_{k}\delta_1=\rho_{k+1}$ for $k\in\{1,\ldots,m-1\}$ and $\widehat \delta_k \delta_1=\delta_k$ for $k\in\{1,\ldots,m\}$.
\end{proof}

\medbreak

Define 
\begin{align*}
p_\eps:=\frac{n+2}{n-2}-\eps. 
\end{align*}
The solution $u_\eps$ satisfies the so-called (pointwise) radial Pohozaev identity (see the proof of \cite[Theorem 2.1]{C10}):
\begin{align}\label{poho1}
\Big(\frac{r^n}{2n}(u_\eps'(r))^2+\frac{n-2}{2n}r^{n-1}u_\eps'(r) u(r) &+\frac{r^n}{n(p_\eps+1)}|u_\eps(r)|^{p_\eps+1} \Big)' \nonumber\\
&=r^{n-1}|u_\eps(r)|^{p_\eps+1} \Big(\frac{1}{p_\eps+1}-\frac{n-2}{2n}\Big)
\end{align}
for $r=|x|\in [0,1]$. Since
\begin{align*}
 \lim_{\eps\to 0}\frac{\frac{1}{p_\eps+1}-\frac{n-2}{2n}}{\varepsilon(\frac{n-2}{2n})^2} =1,\qquad \text{ that is, }\qquad
 2n\eps^{-1}\Big(\frac{1}{p_\eps+1}-\frac{n-2}{2n}\Big)=\frac{(n-2)^2}{2n}+o(1),
\end{align*}
and
\begin{align*}
\frac{2n}{n(p_\eps+1)}=\frac{2}{\frac{2n}{n-2}-\eps}=\frac{n-2}{n-\eps\frac{n-2}{2}},
\end{align*}
then, identity \eqref{poho1} can be rewritten as
\begin{align}\label{eq:pohozaev}
\varepsilon^{-1}\Big(
(u_\eps'(r))^2r^n+(n-2)r^{n-1}u_\eps'(r) u(r)
&+\frac{(n-2)}{n-\varepsilon\frac{n-2}{2}}|u_\eps(r)|^{\frac{2n}{n-2}-\varepsilon}r^n 
\Big)'\nonumber\\
&= \Big(\frac{(n-2)^2}{2n}+o(1)\Big)|u_\eps(r)|^{\frac{2n}{n-2}-\varepsilon}r^{n-1}.
\end{align}
for $r=|x|\in [0,1]$, where $o(1)$ is a function of $\eps$ such that $\lim_{\varepsilon\to 0}o(1)=0$.

\medskip

We use $\chi_A$ to denote the characteristic function of $A\subset \R^N$, that is, $\chi_A(x)=1$ if $x\in A$ and $\chi_A(x)=0$ if $x\not\in A$.  The following result is shown (in more generality) in \cite{DIP17}.

\begin{theo}\label{dip:thm}
For $\eps\in (0,\frac{4}{n-2})$ and $m\geq 2$, let $u_\eps$ be the solution of \eqref{sceq}, \eqref{dbc} with $m-1$ interior zeros such that $u_\eps(\delta_{1,\eps})>0$, and let
  \begin{align*}
   p_\eps:=\frac{n+2}{n-2}-\eps,\qquad z_\varepsilon(x):=\frac{u_\eps\left( u_\eps(\delta_{1,\eps})^{\frac{1-p_\eps}{2}} x \right)}{u(\delta_{1,\eps})}\chi_{\{u_\eps(\delta_{1,\eps})^{\frac{p_\eps-1}{2}}\rho_{2,\eps}<|x|<u_\eps(\delta_{1,\eps})^\frac{p_\eps-1}{2}\}} 
  \end{align*}
for $x\in\R^N$.  There exist $v\in L^{1}(\R^N)$ and $\eps_0(n,m)=\eps_0\in(0,\frac{4}{n-2})$ such that 
\begin{align}\label{dip_aux1}
|z_\eps|^{p_\eps},\  |z_\eps|^{p_\eps+1}\leq v\qquad \text{ in }\R^N\qquad \text{for all }\quad\eps\in (0,\eps_0);
\end{align} 
moreover,
\begin{align}\label{dip_aux2}
 \lim_{\eps \to 0}u_\eps(\delta_{1,\eps})^{\frac{p_\eps-1}{2}}=\infty,\qquad \lim_{\eps\to 0}\delta_{1,\eps} u_\eps(\delta_{1,\eps})^{\frac{p_\eps-1}{2}}=0,
\end{align}
and, for every compact set $K\subset \R^N\backslash\{0\}$,
\begin{align}\label{dip_aux3}
 \lim_{\varepsilon\to 0}\|z_\eps-U\|_{C^2(K)}=0,\qquad \text{where } \quad U(x):=\Big(1+\frac{|x|^2}{n(n-2)}\Big)^{-\frac{n-2}{2}}\quad \text{ for }x\in\R^N.
 \end{align}
\end{theo}
\begin{proof}
Statement \eqref{dip_aux2} is a direct consequence of  \cite[equations (3.20), (3.40)]{DIP17}, while \eqref{dip_aux3} follows from \cite[Theorem 3.7]{DIP17}. Now, by combining Proposition 3.6 and Corollary 3.12 from \cite{DIP17},  we have the existence of $\eps_0(n,m)=\eps_0\in(0,\frac{4}{n-2})$ and $\gamma\in (0,1)$  such that 
\[
|u_\eps(x)|\leq \frac{u(\delta_{1,\eps})}{\left[1+\frac{1}{2n(n-2)} u(\delta_{1,\eps})^{p_\eps-1}|x|^2\right]^\frac{n-2}{2}}\qquad \forall x\in B:\ \gamma^{-\frac{1}{n}} \delta_{1,\eps} \leq |x|<1
\]
Therefore,
\[
|z_\eps(x)|\leq \frac{1}{\left[1+ \frac{1}{2n(n-2)}|x|^2\right]^\frac{n-2}{2}}  \qquad \forall x\in \R^n:\ \gamma^{-\frac{1}{n}} \delta_{1,\eps} u(\delta_{1,\eps})^\frac{p_\eps-1}{2}\leq |x|\leq u(\delta_{1,\eps})^\frac{p_\eps-1}{2}.
\]
On the other hand, since $\delta_{1,\eps}$ it the global maximum of $u_\eps$ on $[\rho_{2,\eps},1]$ and $z_\eps=0$ for $|x|\notin [\rho_{2,\eps} u(\delta_{1,\eps})^\frac{p_\eps-1}{2},u(\delta_{1,\eps})^\frac{p_\eps-1}{2}]$, we have 
\[
|z_\eps|\leq 1 \qquad 
\text{ in }\R^n.
\]
In particular, $|z_\eps|^{p_\eps+1}\leq |z_\eps|^{p_\eps}$. Furthermore, for $|x|>1$ and $\eps$ small, 
\begin{align*}
|z_\eps|^{p_\eps}=|z_\eps|^{\frac{n+2}{n-2}-\eps}\leq \left[1+ \frac{1}{2n(n-2)}|x|^2\right]^{-\frac{n+2}{2}+\eps\frac{n-2}{2}}\leq \left[1+ \frac{1}{2n(n-2)}|x|^2\right]^{-\frac{n+1}{2}}
\end{align*}
and $|z_\eps|^{p_\eps}\leq 1$ for $|x|<1$, so \eqref{dip_aux1} also holds. 
\end{proof}

In the following we use some standard properties of the beta function 
\begin{align*}
B(x,y)=\int_0^1 t^{x-1} (1-t)^{y-1}\, dt=\frac{\Gamma(x)\Gamma(y)}{\Gamma(x+y)}\quad \text{ for }x,y>0,
\end{align*}
where $\Gamma$ is the Gamma function (recall that $\Gamma(x+1)=x\Gamma(x)$ for all $x>0$).
In particular, we use the identity 
\begin{align}\label{EMO}
\int_0^\infty \Big(1+\frac{t^2}{n(n-2)}\Big)^{-y}\, t^{n-1}\, dt
=\frac{1}{2}(n(n-2))^{\frac{n}{2}}
B(\,\frac{n}{2}\,,\,y-\frac{n}{2}\,)
\end{align}
for $y>\frac{n}{2}$; see, for example \cite[page 10 formula (16)]{EMOT81}.

\begin{lemma}\label{2nrD:lemma2}
The profiles $\delta_{1,\eps}$, $|u(\delta_{1,\eps})|$, and $|u_\eps'(1)|$ satisfy that
\begin{align}\label{A1eps}
\lim_{\varepsilon\to 0}(\kappa_n \varepsilon)^{-1}|u_\eps(\delta_{1,\eps})|^{-\varepsilon\frac{n-2}{2}}\Big(
(u_\eps'(1))^2
-\frac{n-2}{n-\varepsilon\frac{n-2}{2}}&|u_\eps(\delta_{1,\eps})|^{\frac{2n}{n-2}-\varepsilon}\delta_{1,\eps}^n 
\Big)=  n^\frac{n-2}{2}(n-2)^\frac{n+2}{2}.\quad 
\end{align}
\end{lemma}
\begin{proof}  To ease notation with omit some $\eps$ subindices. Recall the definitions of $z_\eps$ and $U$ from Theorem \ref{dip:thm}, and that $p_\eps:=\frac{n+2}{n-2}-\eps$. Assume without loss of generality that $u(\delta_1)>0$. Integrating \eqref{eq:pohozaev} in $(\delta_1,1)$ we have that
\begin{align}
\varepsilon^{-1}\Big(
(u'(1))^2
-\frac{n-2}{n-\varepsilon\frac{n-2}{2}}u(\delta_1)^{\frac{2n}{n-2}-\varepsilon}(\delta_1)^n 
\Big)
= \Big(\frac{(n-2)^2}{2n}+o(1)\Big)\int_{\delta_1}^1 u(r)^{\frac{2n}{n-2}-\varepsilon}r^{n-1}\ dr.\label{poho11}
\end{align}
Using the change of variables $r=u(\delta_1)^\frac{1-p_\eps}{2}t$, Theorem \ref{dip:thm}, dominated convergence, and \eqref{EMO}, we have
\begin{align}
\lim_{\eps\to0}|u(\delta_1)|^{-\varepsilon\frac{n-2}{2}}&\int_{{\delta_1}}^1u(r)^{p_\eps+1} r^{n-1}\ dr\nonumber\\
& =\lim_{\eps\to0}\int_{{\delta_1}u(\delta_1)^\frac{p_\eps-1}{2}}^{u(\delta_1)^\frac{p_\eps-1}{2}}\left|\frac{u\left(u(\delta_1\right)^\frac{1-p_\eps}{2}t)}{u(\delta_1)}\right|^{p_\eps+1} t^{n-1}u(\delta_1)^{\frac{1-p_\eps}{2}n+p_\eps+1-\varepsilon\frac{n-2}{2}}\ dt\nonumber\\
&=\lim_{\eps\to0}\int_{{\delta_1}u(\delta_1)^\frac{p_\eps-1}{2}}^{u(\delta_1)^\frac{p_\eps-1}{2}}|z_\eps(t)|^{p_\eps+1} t^{n-1}\ dt \nonumber\\
&=\int_{0}^\infty|U(t)|^{\frac{2 n}{n-2}} t^{n-1}\ dt=\frac{((n-2) n)^{\frac{n}{2}} \Gamma \left(\frac{n}{2}\right)^2}{2 \Gamma (n)}.\label{poho2}
\end{align}
But then, by \eqref{poho11} and \eqref{poho2},
\begin{align*}
\lim_{\varepsilon\to 0}\varepsilon^{-1}|u(\delta_1)|^{-\varepsilon\frac{n-2}{2}}\Big(
(u'(1))^2
-\frac{(n-2)}{n-\varepsilon\frac{(n-2)}{2}}|u(\delta_1)|^{\frac{2n}{n-2}-\varepsilon}\delta_1^n 
\Big)
= \frac{(n(n-2))^{2+\frac{n}{2}}}{n^3}  \frac{\Gamma \left(\frac{n}{2}\right)^2}{4 \Gamma (n)},
\end{align*}
and the claim follows by recalling the definition of $\kappa_n$ in \eqref{Cchi}.
\end{proof}

\begin{lemma}\label{L1}
The sequences $|u_\eps(\delta_{1,\eps})|$ and $|u_\eps'(1)|$ satisfy that
\begin{align}\label{A2eps}
\lim_{\eps\to 0} |u_\eps(\delta_{1,\eps})|^{1-\frac{\varepsilon(n-2)}{2}} |u_\eps'(1)|= \frac{((n-2) n)^{\frac{n}{2}}}{n}.
\end{align}
\end{lemma}
\begin{proof}
To ease notation we omit some $\eps$ subindices. Recall that $p_\eps=\frac{n+2}{n-2}-\varepsilon$ and assume without loss of generality that $u(\delta_1)>0$.  Integrating in $(\delta_1,1)$ the identity
\begin{align*}
 -(u'(r) r^{n-1})' = |u(r)|^{p_\eps-1}u(r) r^{n-1}\quad \text{ in }(0,1),\qquad u'(0)=u(1)=0,
\end{align*}
using the change of variables $t=r u(\delta_1)^\frac{p_\eps-1}{2}$, and the notation from Theorem \ref{dip:thm}, we have that
\begin{multline*}
|u'(1)|=\int_{\delta_1}^1 |u(r)|^{p_\eps} r^{n-1}\ dr = u(\delta_1)^\frac{(1-p_\eps)n}{2}\int_{\delta_1 u(\delta_1)^\frac{p_\eps-1}{2}}^{u(\delta_1)^\frac{p_\eps-1}{2}} \left|u\left( u(\delta_1)^{\frac{1-p_\eps}{2}} t\right)\right|^{p_\eps} t^{n-1}\ dt\\
=u(\delta_1)^{p_\eps+\frac{(1-p_\eps)n}{2}}\int_{\delta_1 u(\delta_1)^\frac{p_\eps-1}{2}}^{u(\delta_1)^\frac{p_\eps-1}{2}} |z_\eps(t)|^{p_\eps} t^{n-1}\ dt=u(\delta_1)^{-1+\frac{\varepsilon(n-2)}{2}}\int_{\delta_1 u(\delta_1)^\frac{p_\eps-1}{2}}^{u(\delta_1)^\frac{p_\eps-1}{2}} |z_\eps(t)|^{p_\eps} t^{n-1}\ dt,
\end{multline*}
then, by Theorem \ref{dip:thm}, dominated convergence, and \eqref{EMO},
\begin{align}\label{lim}
\lim_{\eps\to 0} u(\delta_1)^{1-\frac{\varepsilon(n-2)}{2}} |u'(1)|= \int_{0}^{\infty} |U(t)|^{\frac{n+2}{n-2}} t^{n-1}\ dt=  \frac{((n-2) n)^{\frac{n}{2}}}{n},
\end{align}
as claimed.
\end{proof}

\medbreak

\begin{proof}[Proof of Proposition \ref{prop2}]  Arguing as in the proof of Proposition \ref{prop1}, using \eqref{A1eps}, \eqref{A2eps} and \eqref{A0} for $k=1$, we have
\begin{align}
(\kappa_n \varepsilon)^{-1}|u_\eps(\delta_{1,\eps})|^{-\varepsilon\frac{n-2}{2}}\Big(
(u_\eps'(1))^2
-\frac{2(n-2)}{2n-\varepsilon(n-2)}|u_\eps(\delta_{1,\eps})|^{\frac{2n}{n-2}-\varepsilon}\delta_{1,\eps}^n 
\Big)
&=  n^\frac{n-2}{2}(n-2)^\frac{n+2}{2} + o(1) \label{eq:proofProp2}
\end{align}
and
\begin{equation}\label{eq:proofProp22}
\begin{aligned}
|u_\eps(\delta_{1,\eps})|^{1-\frac{\varepsilon(n-2)}{2}} |u_\eps'(1)|&= n^\frac{n-2}{2}(n-2)^{\frac{n}{2}} + o(1),\\
\delta_1^\frac{2(n-2)}{4-\eps(n-2)}|u_\eps(\delta_{1,\eps})|(\kappa_n \eps)^\frac{2-n}{2n} &= \widetilde D(1,m)+o(1).
\end{aligned}
\end{equation}
From \eqref{eq:proofProp22} we see that
\[
|u_\eps(\delta_{1,\eps})|=(\kappa_n \eps)^\frac{n-2}{2n}\delta_{1,\eps}^\frac{2(2-n)}{4-\eps(n-2)}(\widetilde D(1,m)+o(1))
\]
and
\[
\begin{aligned}
|u_\eps'(1)|&=|u_\eps(\delta_{1,k})|^\frac{\eps(n-2)-2}{2}(n^\frac{n-2}{2}(n-2)^\frac{n}{2}+o(1))\\
		&=(\kappa_n \eps)^{\frac{n-2}{4n}(\eps(n-2)-2)}\delta_{1,\eps}^{(n-2)\frac{2-\eps(n-2)}{4-\eps(n-2)}}(\widetilde D(1,m)+o(1))^\frac{\eps(n-2)-2}{2}(n^\frac{n-2}{2}(n-2)^\frac{n}{2}+o(1)).
\end{aligned}
\]
Replacing these expressions in \eqref{eq:proofProp2} and using \eqref{sim}, we have
\begin{multline*}
(\kappa_n \eps)^{-1}\left(  
 (\kappa_n \eps)^{-\frac{n-2}{n}}\delta_1^{\frac{4(n-2)-\eps(n-2)^2}{4-\eps(n-2)}}  ( \widetilde D(1,m)^{-2} n^{n-2}(n-2)^n + o(1))   \right. \\
 \left. -  \frac{n-2}{n}(\kappa_n \eps) (\widetilde D(1,m)^\frac{2n}{n-2}+o(1))\right)=n^\frac{n-2}{2}(n-2)^\frac{n+2}{2}+o(1).
\end{multline*}
Thus $\delta_{1,\eps}=A \eps^{\frac{a+b\eps}{c+d\eps}}+o(1)$ for some $A,a,b,c,d\in \R$ and, again by \eqref{sim},
we can find by direct substitution that
\begin{align*}
\lim_{\eps\to 0}\delta_{1,\eps} (\kappa_n\epsilon)^{-\frac{2 (n-1)}{(n-2) n}}
&=(n-2)^{\frac{1}{2-n}-1} n^{\frac{1}{2-n}-1} \widetilde D(1,m)^{\frac{2}{n-2}} \left(\widetilde D(1,m)^{\frac{2 n}{n-2}}+(n-2)^{n/2} n^{n/2}\right)^{\frac{1}{n-2}}\\
&=d(1,m).
\end{align*}
Using this rate and \eqref{eq:proofProp22} we can now obtain similarly $|u_\eps(\delta_{1,\eps})|$ and $|u_\eps'(1)|=|u_\eps'(\rho_{1,\eps})|$, obtaining
\begin{align*}
 \lim_{\eps\to 0}(\kappa_n\epsilon)^\frac{1}{2} |u_\eps(\delta_{1,\eps})|&=\widetilde D(1,m) d(1,m)^{1-\frac{n}{2}}=D(1,m),\\
 \lim_{\eps\to 0}|u_\eps'(\rho_{1,\eps})|(\kappa_n\epsilon)^{-\frac{1}{2}}&= \frac{(n-2)^{\frac{n}{2}} n^{\frac{n}{2}-1} d(1,m)^{\frac{n}{2}-1}}{\widetilde D(1,m)}=Z(1,m).
\end{align*}

To determine the rest of the unknown rates we use \eqref{A0}--\eqref{A03}, namely,
\begin{align*}
& \lim_{\eps\to 0}(\kappa_n\epsilon)^{\frac{2k-1}{2}} |u_\eps(\delta_{k,\eps})|= d(1,m)^{\frac{2-n}{2}} \widetilde D(k,m)=D(k,m) \ \ \text{ for } k\in \{2,\ldots,m\},\\
& \lim_{\eps\to 0}\delta_{k,\eps} (\kappa_n\epsilon)^{-\frac{2 (k n-1)}{(n-2) n}} = d(1,m) \widetilde d(k,m)=d(k,m)\ \ \text{ for } k\in \{2,\ldots,m-1\},\ m\geq 3,\\
&\lim_{\eps\to 0}(\kappa_n \epsilon)^{\frac{2 k n-3 n+2}{2 n-4}} |u_\eps'(\rho_{k,\eps})|= d(1,m)^{-\frac{n}{2}} \widetilde Z(k-1,m)=Z(k,m) \ \ \text{ for } k\in \{2,\ldots,m\},\\
& \lim_{\eps\to 0}\rho_k (\kappa_n\epsilon)^{-\frac{2 (k-1)}{n-2}}= d(1,m)\widetilde z(k-1,m)=z(k,m) \ \ \text{ for } k\in \{2,\ldots,m\}.
\end{align*}
This ends the proof.
\end{proof}

\medbreak

\begin{proof}[Proof of Theorems \ref{explicit:thm_Dirichlet} and \ref{explicit:thm_Neumann}]
 The claims follow from Theorem \ref{1nD:thm} and Propositions \ref{prop1} and \ref{prop2}.
\end{proof}

\section{Remarks and consequences}\label{RC:sec}

\begin{proof}[Proof of Theorem \ref{pointwise:coro}]
Fix $a\in (0,1)$ and let $u_\eps$ be a radial solution of \eqref{sceq} with $m-1$ interior zeros.  Assume first that $m\geq 2$ and $u_\eps$ satisfies Neumann b.c. \eqref{nbc}. Then, by Lemma \ref{prop:lemma}, $\frac{|u_\eps|}{|u_\eps(1)|}\to 1$ uniformly in [a,1] as $\eps\to 0$, and, by Theorem \ref{explicit:thm_Neumann}, $\lim_{\eps\to 0}|u_\eps(1)| (\kappa_n\eps)^{\frac{2-n}{2n}} =\widetilde D(1,m)$ and thus
\begin{align*}
|u_\eps(1)|^{-1} =\frac{(\kappa_n\epsilon)^{\frac{2-n}{2n}}}{\widetilde D(1,m)+o(1)}
\qquad \text{ and }\qquad 
\lim_{\eps\to 0}|u_\eps(x)|\frac{(\kappa_n\epsilon)^{\frac{2-n}{2n}}}{\widetilde D(1,m)+o(1)}=\lim_{\eps\to 0}\frac{|u_\eps(x)|}{|u_\eps(1)|}=1,
\end{align*}
for $x\in[a,1]$. Then, $|u_\eps(x)|(\kappa_n\epsilon)^{\frac{2-n}{2n}}=\widetilde D(1,m)+o(1)
=(n(n-2))^{\frac{n-2}{4}} 
(m-1)^{\frac{n-2}{2n}}
+o(1)
$, with  $o(1)\to 0$ uniformly in $[a,1]$ as $\eps\to 0$.  

Now, let $m\geq 1$, fix $a\in (0,1)$, $x\in[a,1]$, and let $u_\eps$ satisfy Dirichlet b.c. \eqref{dbc}. Since $\delta_{1,\eps}\to 0$ as $\eps\to 0$, we may assume w.l.o.g. that $a>\delta_{1,\eps}$ and $u>0$ in $(\delta_{1,\eps},1)$.  To ease notation we omit some subindices $\eps$ and recall that $p_\eps:=\frac{n+2}{n-2}-\eps$. We now argue as in Lemma \ref{L1}. Note that
\begin{align*}
 -(u'(r) r^{n-1})' = |u(r)|^{p_\eps-1}u(r) r^{n-1}\quad \text{ in }(0,1),\qquad u'(0)=u(1)=0,
\end{align*}
Then, for $r>\delta_{1}$,
\begin{align}\label{reg}
 -u'(r) r^{n-1} = \int_{\delta_{1}}^ru(t)^{p_\eps} t^{n-1}\ dt,
\end{align}
but then, integrating in $(|x|,1)$ and using the change of variables $s=t u(\delta_{1})^\frac{p_\eps-1}{2}$
\begin{align*}
 u(|x|) & = \int_{|x|}^1r^{1-n}\int_{\delta_{1}}^r u(t)^{p_\eps} t^{n-1}\ dt\ dr\\
 & = u(\delta_{1})^{-1+\frac{\eps(n-2)}{2}}\int_{|x|}^1r^{1-n}\int_{\delta_{1}u(\delta_{1})^\frac{p_\eps-1}{2}}^{ru(\delta_{1})^\frac{p_\eps-1}{2}} \left[ \frac{u(s u(\delta_{1})^\frac{1-p_\eps}{2})}{u(\delta_1) }\right]^{p_\eps} s^{n-1}\ ds\ dr.
\end{align*}
Then, by dominated convergence, Theorem \ref{dip:thm}, equation \eqref{lim} and Theorem \ref{explicit:thm_Dirichlet},
\begin{align}\label{ec1}
 \lim_{\eps\to 0}|u_\eps(x)||u(\delta_{1,\eps})|=\frac{((n-2)n)^\frac{n}{2}}{n}
 \int_{|x|}^1y^{1-n}\ dy
 =((n-2)n)^{\frac{n}{2}-1}
 (|x|^{2-n}-1).
 \end{align} 
Since, by Theorem \ref{explicit:thm_Dirichlet},
 \begin{align}\label{ec2}
\lim_{\eps\to 0}   |u_\eps(\delta_{1,\eps})|\ (\kappa_n\epsilon)^{\frac{1}{2}}
      &=D(1,m)=(n(n-2))^\frac{n-2}{4}m^{-\frac{1}{2}}
 \end{align}
we have, by \eqref{ec1} and \eqref{ec2}, 
\begin{align}\label{aa1}
 |u_\eps(x)|(\kappa_n\epsilon)^{-\frac{1}{2}}=\frac{((n-2)n)^{\frac{n}{2}-1}}{D(1,m)}(|x|^{2-n}-1)+ o(1)=(n-2)n)^{\frac{n-2}{4}}m^{\frac{1}{2}}(|x|^{2-n}-1)+o(1),
 \end{align} 
where $o(1)\to 0$ pointwisely as $\eps\to 0$.  We now argue that this convergence is, in fact, uniform in $[a,1]$. Let $\eps_0>0$ be such that
\begin{align}\label{eps0}
0<u_\eps(a)<1\quad \text{ for all }\eps\in(0,\eps_0), 
\end{align}
 then, by \eqref{aa1} and since $u_\eps$ is decreasing in $[a,1]$ (because $a>\delta_{1,\eps}$), we have that $|u_\eps\epsilon^{-\frac{1}{2}}|<C$ in $[a,1]$ for some constant $C(m,n)=C>0$ independent of $\eps$. Then, by \eqref{eps0},  $\epsilon^{-\frac{1}{2}}u_\eps^{p_\eps}\leq \epsilon^{-\frac{1}{2}}u_\eps^{p_\eps}(a)\leq \epsilon^{-\frac{1}{2}}u_\eps(a)\leq C$ in $[a,1]$ for every $\eps\in(0,\eps_0)$, and thus, by \eqref{reg}, it follows that $\epsilon^{-\frac{1}{2}}\|u_\eps\|_{C^1[a,1]}$ is uniformly bounded; but then, by Arzel\`{a}-Ascoli theorem, \eqref{aa1} holds with $o(1)\to 0$ uniformly in $[a,1]$.
\end{proof}

\medbreak

\begin{proof}[Proof of Corollary \ref{dp:coro}]
In \cite[Theorem 1]{CD06} (see also \cite{PW07}) the following result is shown: given $n\geq 3$ and $m\geq 1$, for all sufficiently small $\eps>0$, there is a radial solution $u_\eps$ of \eqref{sceq}, \eqref{dbc} with exactly $m-1$ nodal spheres in $(0,1)$ which has the form \eqref{dpf} for some $f_\eps$ which is uniformly bounded in $B$ and where $\gamma_n=(n(n-2))^\frac{n-2}{4}$ and $\alpha_k$ are some positive constants for $k\in\{1,\ldots,m\}.$  

Let $\delta_{m,\eps}=0<\delta_{m-1,\eps}<\ldots< \delta_{1,\eps}<1$ be the sequence of critical points of $u_\eps$ and fix $j\in\{1,\ldots,m-1\}$. Then,  by Theorem \ref{explicit:thm_Dirichlet},
\begin{align*}
 \delta_{j,\eps}&=(d(j,m)+o(1))(\kappa_n\eps)^{\frac{2 (j n-1)}{n(n-2)}}.\quad  
\end{align*}
Let $d_\eps:=\frac{\delta_{j,\eps}}{\eps^{\frac{2 (j n-1)}{n(n-2)}}}= (d(j,m)+o(1))\kappa_n^{\frac{2 (j n-1)}{n(n-2)}}$ and, for $k\in\{1,\ldots,m\}$,
\begin{align*}
 A_{k,j,\eps}:=\Bigg(
 \frac{1}{1+[\alpha_k \eps^{\frac{1}{2}-k}]^\frac{4}{n-2}d_\eps^2\, \eps^{\frac{4(j n-1)}{n(n-2)}}}
 \Bigg)^\frac{n-2}{2} \eps^{j-k}
 =\Bigg(
 \frac{1}{1+\alpha_k^\frac{4}{n-2}d_\eps^2\, \eps^{(\frac{2}{n-2})(\frac{n-2+2n(j-k)}{n})}}
 \Bigg)^\frac{n-2}{2} \eps^{j-k}.
\end{align*}
Observe that $A_{j,j,\eps}=1+o(1)$, $A_{k,j,\eps}=o(1)$ if $j>k$, and 
\begin{align*}
 A_{k,j,\eps}
 =
 \Bigg(\frac{1}{\eps^{(\frac{2}{n-2})(\frac{2}{n}-1+2(k-j))}+\alpha_k^\frac{4}{n-2}d_\eps^2}\Bigg)^\frac{n-2}{2}
 \eps^{(-1+\frac{2}{n}+k-j)}=o(1)\qquad \text{ if }k>j,
\end{align*}
since $-1+\frac{2}{n}+k-j>0$. Then, by \eqref{dpf} and since $\eps^{j-\frac{1}{2}}f_\eps(d_\eps \eps^{\frac{2 (j n-1)}{n(n-2)}})\eps^\frac{1}{2}=o(1)$,
\begin{align*}
 \eps^{j-\frac{1}{2}}u_\eps(\delta_{j,\eps})=\eps^{j-\frac{1}{2}}
 u_\eps(d_\eps \eps^{\frac{2 (j n-1)}{n(n-2)}})&=\gamma_n \sum_{k=1}^m(-1)^{k+1}A_{k,j,\eps}
 \alpha_k-\eps^{j}f_\eps(d_\eps \eps^{\frac{2 (j n-1)}{n(n-2)}})\\
& = (-1)^{j+1}\gamma_n\alpha_j-o(1).
\end{align*}
On the other hand, by Theorem \ref{explicit:thm_Dirichlet} we know that 
$\lim_{\eps\to 0}(\kappa_n\eps)^{j-\frac{1}{2}}|u_\eps(\delta_{j,\eps})|=D(j,m)$ and thus
\begin{align*}
\alpha_j=\gamma_n^{-1}\lim_{\eps\to 0}\eps^{j-\frac{1}{2}}|u_\eps(\delta_{j,\eps})|=\gamma_n^{-1}\kappa_n^{\frac{1}{2}-j}D(j,m)
=\kappa_n^{\frac{1}{2}-j}\frac{ \Gamma (m-j+1)}{m^\frac{1}{2} \Gamma (m)} \text{ for } j\in \{1,\ldots,m-1\}.
\end{align*}
Finally, for $j=m$,
\begin{align*}
 \eps^{m-\frac{1}{2}}u_\eps(\delta_{m,\eps})&=\eps^{m-\frac{1}{2}}
 u_\eps(0)=\gamma_n \sum_{k=1}^m(-1)^{k+1}
 \alpha_k \eps^{m-k}- \eps^{m-\frac{1}{2}}f_\eps(\delta_{m,\eps})\eps^\frac{1}{2}
 = (-1)^{m+1}\gamma_n\alpha_m+o(1).
\end{align*}
which, combined with $\lim_{\eps\to 0}(\kappa_n\eps)^{m-\frac{1}{2}}|u_\eps(\delta_{m,\eps})|=D(m,m)+o(1)$ provides
\[
\alpha_m=\gamma_n^{-1}\lim_{\eps\to 0}\eps^{m-\frac{1}{2}}|u_\eps(\delta_{m,\eps})|=\gamma_n^{-1}\kappa_n^{\frac{1}{2}-m}D(m,m)
=\kappa_n^{\frac{1}{2}-m}m^{-\frac{1}{2}} \Gamma (m)^{-1}.
\]
Then \eqref{dpf} holds with $\alpha_k$ as in \eqref{alphak}. It remains to argue \eqref{O}. Let $K$ be a compact subset of $\overline{B}\backslash \{0\}$. By Theorem~\ref{pointwise:coro},
 \begin{align}\label{u1}
 \lim_{\eps\to 0}\| |u_\eps|(\kappa_n\epsilon)^{-\frac{1}{2}} - \gamma_nm^\frac{1}{2}(|\cdot|^{2-n}-1)\|_{L^\infty(K)}=0.
  \end{align}
  Moreover, for $y\in K$,
  \begin{align}
   \varphi_\eps(y)&:=(\kappa_n\eps)^{-\frac{1}{2}}\gamma_n \sum_{k=1}^m(-1)^{k+1}\Bigg(\frac{1}{1+[\alpha_k \eps^{\frac{1}{2}-k}]^\frac{4}{n-2}|y|^2}\Bigg)^\frac{n-2}{2}\alpha_k \eps^{\frac{1}{2}-k}\nonumber\\
   &=\kappa_n^{-\frac{1}{2}}\gamma_n \sum_{k=1}^m(-1)^{k+1}\Bigg(\frac{1}{\eps^{(k-\frac{1}{2})(\frac{4}{n-2})}+\alpha_k^\frac{4}{n-2}|y|^2}\Bigg)^\frac{n-2}{2}\alpha_k \eps^{k-1}\nonumber
   \\&=\gamma_n\kappa_n^{-\frac{1}{2}}\alpha_1^{-1}|y|^{2-n}+o(1)
   =\gamma_nm^\frac{1}{2}|y|^{2-n}+o(1),\label{u2}
  \end{align}
where $o(1)\to 0$ uniformly in $K$ as $\eps\to 0$. Observe that $u_\eps>0$ in $K$ for sufficiently small $\eps>0$; then, by \eqref{dpf}, \eqref{u1}, and \eqref{u2}, we have
\begin{align*}
&\lim_{\eps\to 0}\|\kappa_n^{-\frac{1}{2}}f_\eps-\gamma_nm^\frac{1}{2}\|_{L^\infty(K)}=\lim_{\eps\to 0}\|\varphi_\eps-(\kappa_n\eps)^{-\frac{1}{2}}u_\eps-\gamma_nm^\frac{1}{2}\|_{L^\infty(K)}\\
&\leq \lim_{\eps\to 0}\|\gamma_nm^\frac{1}{2}(|\cdot|^{2-n}-1)-(\kappa_n\eps)^{-\frac{1}{2}}u_\eps\|_{L^\infty(K)}+
\lim_{\eps\to 0}\|\varphi_\eps- \gamma_n m^\frac{1}{2}|\cdot|^{2-n}\|_{L^\infty(K)}=0.  \qedhere
\end{align*} 
\end{proof}
 
 \begin{proof}[Proof of Corollary \ref{nbt:coro}]
 Let $u_\eps$ be a (Neumann) solution of \eqref{sceq}, \eqref{nbc} with $m-1$ interior zeros such that $(-1)^{m+1}u_\eps(0)>0$ and let $v_\eps$ be a (Dirichlet) solution of \eqref{sceq}, \eqref{nbc} with $m-1$ interior zeros and such that $(-1)^{m+1}v_\eps(0)>0$.  Then, by Corollary \ref{dp:coro},
\begin{align}\label{dpf1}
v_\eps(y)=\gamma_n \sum_{k=1}^m(-1)^{k+1} \Bigg(\frac{1}{1+[\alpha_k \eps^{\frac{1}{2}-k}]^\frac{4}{n-2}|y|^2}\Bigg)^\frac{n-2}{2}\alpha_k \eps^{\frac{1}{2}-k}-f_\eps(y)\eps^\frac{1}{2}
\end{align}
for $y\in B$, where $f_\eps:B\to\R$ is uniformly bounded in $B$ and $\alpha_k$ is given by \eqref{alphak}.  Denote by $(\delta_{k,\eps})_{k=1}^m$ the decreasing sequence of all critical points of $v_\eps$.
By uniqueness and Theorem \ref{explicit:thm_Dirichlet}, we have that
\begin{align}\label{dpf2}
u_\eps(x)=\delta_{1,\eps}^\frac{2(n-2)}{4-\eps(n-2)} v_\eps(\delta_{1,\eps} x)
=(1+o(1))\delta_{1,\eps}^\frac{n-2}{2} v_\eps(\delta_{1,\eps} x),
\end{align}
where $\delta_{1,\eps}=(1+o(1))d(1,m) (\kappa_n \eps)^\frac{2(n-1)}{n(n-2)}$ and $d(1,m)=(m-1)^{\frac{1}{n}} m^{\frac{1}{n-2}}$. To ease notation, denote
\begin{align*}
 d_\eps:=\left(\frac{\delta_{1,\eps}}{\eps^\frac{2(n-1)}{n(n-2)}}\right)^\frac{2(n-2)}{4-\eps(n-2)}=[(1+o(1))d(1,m)\kappa_n^\frac{2(n-1)}{n(n-2)}]^\frac{2(n-2)}{4-\eps(n-2)}
 =(1+o(1))
 (m-1)^{\frac{n-2}{2n}} m^{\frac{1}{2}}
 \kappa_n^\frac{n-1}{n}
 .
\end{align*}
 Observe that
\begin{equation}\label{futurereference}
d_\eps\alpha_k=(1+o(1))\beta_k=:\beta_{k,\eps}\qquad \text{ for }k\in\{0,1,\ldots,m\}
\end{equation}
where $\lim_{\eps\to 0}\beta_{k,\eps}=\beta_k$ and $\beta_k$ is given by \eqref{beta}.  Then, using \eqref{dpf1}, \eqref{dpf2}, and \eqref{futurereference}, we have, after some calculations,
\begin{align*}
u_\eps(x)&=d_\eps \eps^\frac{n-1}{n}\gamma_n \sum_{k=1}^m(-1)^{k+1} \Bigg(\frac{1}{1+[\alpha_k d_\eps\eps^{\frac{n-2}{2n}-(k-1)}]^\frac{4}{n-2}|x|^2}\Bigg)^\frac{n-2}{2}\alpha_k \eps^{\frac{1}{2}-k}-f_\eps(\delta_{1,\eps}x) d_\eps\eps^\frac{3n-2}{2n}\\
&=\gamma_n \sum_{k=1}^m(-1)^{k+1} \Bigg(\frac{1}{1+[\beta_{k,\eps}\, \eps^{\frac{n-2}{2n}-(k-1)}]^\frac{4}{n-2}|x|^2}\Bigg)^\frac{n-2}{2}\beta_{k,\eps}\, \eps^{\frac{n-2}{2n}-(k-1)}+g_\eps(x)\eps^\frac{3n-2}{2n},
\end{align*}  
where $g_\eps:=-d_\eps f_\eps(\delta_{1,\eps}\cdot)$ is a function which is uniformly bounded in $B$.
 \end{proof}

\medbreak

\begin{proof}[Proof of Corollary \ref{ODE:coro}] 
The existence of $w_\eps$ follows from standard ODE considerations, see for example \cite[page 294]{N83}. Note that
\begin{align}\label{u:1}
u_{\eps}(x)
= r_{m,\eps}^{\frac{2(n-2)}{4-\eps(n-2)}} w_\eps(r_{m,\eps} x)\qquad \text{ for }x\in\overline{B} 
\end{align}
is a radial solution of \eqref{sceq}, \eqref{dbc} with $(m-1)$-interior zeros.  Then, using Theorem \ref{explicit:thm_Dirichlet} applied to \eqref{u:1} and recalling that $\delta_{m,k}=0$ and $\rho_{1,k}=1$ therein, and $w_\eps(0)=1$, we have 
\begin{align*}
\lim_{\eps\to 0} r_{m,\eps}^\frac{2(n-2)}{4-\eps(n-2)} (\kappa_n \eps)^\frac{2m-1}{2}&=D(m,m),\\
\lim_{\eps\to 0} r_{m,\eps}^\frac{2n-\eps(n-2)}{4-\eps(n-2)}|w'_\eps(r_{m,\eps})|(\kappa_n \eps)^{-\frac{1}{2}} &=Z(1,m).
\end{align*}
Thus, taking into consideration \eqref{sim}, we have that
\begin{align*}
\lim_{\eps\to 0}r_{m,\epsilon }(\kappa_n\epsilon)^{\frac{2m-1}{n-2}}&=D(m,m)^{\frac{2}{n-2}}
= \sqrt{n-2} \sqrt{n}\ m^{\frac{1}{2-n}}\  \Gamma (m)^{-\frac{2}{n-2}},\\
\lim_{\eps\to0}|w_\eps'(r_{m,\epsilon })|(\kappa_n\epsilon)^{\frac{1-m n}{n-2}}&= Z(1,m) D(m,m)^{-\frac{n}{n-2}}
=\sqrt{\frac{n-2}{n}} m^{\frac{n-1}{n-2}}  \Gamma (m)^{\frac{n}{n-2}}.
\end{align*}

Similarly, let $s_{m,\eps}>0$ be the $m$-th critical point of $w_\eps$ for $m\geq 2$ (the first critical point is always the origin). Then, 
\begin{align*}
v_{\eps}(x)
= s_{m,\eps}^{\frac{2(n-2)}{4-\eps(n-2)}} w_\eps(s_{m,\eps} x)\qquad \text{ for }x\in\overline{B}
\end{align*}
is a radial solution of \eqref{sceq}, \eqref{nbc} with $(m-1)$-interior zeros, and therefore, by Theorem \ref{explicit:thm_Neumann} (recalling that $\delta_{1,\eps}=1$ and $\delta_{m,\eps}=0$ therein) and since $w_\eps(0)=1$, we have
\begin{align*}
\lim_{\eps \to 0} s_{m,\eps}^\frac{2(n-2)}{4-\eps(n-2)} (\kappa_n \eps)^\frac{2mn-3n+2}{2n}&=\widetilde D(m,m),\\
\lim_{\eps\to 0} s_{m,\eps}^\frac{2(n-2)}{4-\eps(n-2)}|w_\eps(s_{m,\eps})| (\kappa_n \eps)^\frac{2-n}{2n}&=\widetilde D(1,m).
\end{align*}
Again, by  \eqref{sim}, we conclude that
\begin{align*}
\lim_{\eps\to 0}s_{m,\eps}\ (\kappa_n\epsilon)^{\frac{2mn-3n+2}{n(n-2)}}&= \widetilde D(m,m)^{\frac{2}{n-2}}=\sqrt{n-2} \sqrt{n} (m-1)^{\frac{1}{n}} \Gamma (m)^{-\frac{2}{n-2}},\\
\lim_{\eps\to 0}|w_\eps(s_{m,\eps})|\ (\kappa_n\epsilon)^{1-m} &=\frac{\widetilde D(1,m)}{\widetilde D(m,m)}=\Gamma (m). \qedhere
\end{align*}
\end{proof}

\begin{remark} 
Similarly as in Corollaries \ref{dp:coro} and \ref{nbt:coro}, one can use Corollary \ref{ODE:coro} and the ansatz \eqref{dpf} to construct a function which approximates the solution of the subcritical problem \eqref{ws}. For instance, let $n\geq 3$, $\eps\in(0,\frac{4}{n-2})$, $w_\eps$ a solution of \eqref{eq:radial}, $M\in\N$, and let
\begin{align}\label{ib}
v_\eps(y):=&\sum_{m=1}^M(-1)^{m+1} \Bigg(\frac{1}{1+[\gamma_n^{-1}\Gamma(m)(\kappa_n\eps)^{m-1}]^\frac{4}{n-2}|y|^2}\Bigg)^\frac{n-2}{2}\Gamma(m)\, (\kappa_n\eps)^{m-1}\ \ \text{ for }y\geq 0,
  \end{align}
where $\gamma_n$ is given in \eqref{bubble} and $\kappa_n$ in \eqref{Cchi}.  Let $0=s_{1,\eps}<s_{2,\eps}<\ldots$ denote the increasing sequence of all the critical points of $w_\eps$. Then
\begin{align}
&\lim_{\eps\to 0}|\epsilon^{1-k}(|w_\eps(s_{k,\eps})|-|v_\eps(s_{k,\eps})|)|=0\qquad \text{ for all }k\in\{1,\ldots,M\},\label{cpib}\\
&\lim_{\eps\to 0}\|\epsilon^{1-M}(|w_\eps|+|v_\eps|)\|_{L^\infty(s_{M,\eps},\infty)}\leq 2\Gamma(M)\kappa_n^{M-1}.\label{cpib2}
\end{align}
We leave the details of the proof to the interested reader.
\end{remark}

\bigbreak

\textbf{Acknowledgments.} 
H. Tavares is partially supported by ERC Advanced Grant 2013 n. 339958 ``Complex
Patterns for Strongly Interacting Dynamical Systems - COMPAT'' and by FCT (Portugal) grant UID/MAT/04561/2013. 
A. Salda\~{n}a was supported by the Alexander von Humboldt Foundation (Germany), and is also grateful to the Universit\`{a} di Roma La Sapienza and to the Universidade de Lisboa, where this project started during two short research stays.


\end{document}